\documentclass[generic]{imsart}

\usepackage[latin1]{inputenc}  
\usepackage[english]{babel}
\usepackage{amsmath, amssymb, amsthm}
\usepackage{bbm}
\usepackage{enumerate}
\usepackage{color}
\usepackage{booktabs}
\usepackage{multicol,multirow,tabularx}
\usepackage{subfigure}
\usepackage[arrow, matrix, curve]{xy}

\usepackage{stmaryrd}
\usepackage{tikz}
\usepackage{tikz-cd}

\newcommand{\id}{{1\hspace{-1mm}{\rm I}}}
\newcommand{\C}{\mathbb{C}}
\newcommand{\N}{\mathbb{N}}
\newcommand{\Z}{\mathbb{Z}}

\newcommand{\R}{\mathbb{R}}
\newcommand{\T}{\mathbb{T}}

\newcommand{\F}{\mathcal{F}}
\newcommand{\M}{\mathcal{M}}

\newcommand{\B}{\mathcal{B}}
\DeclareMathOperator*{\plim}{\textup{P-lim}}

\DeclareMathOperator{\E}{\mathbb{E}}

\DeclareMathOperator*{\argmin}{\text{argmin}}
\newcommand{\supp}{\textup{supp}}

\newcommand{\Gi}{\mathcal{G}}

\newcommand{\Imag}{\textup{Im}}

\newcommand{\dx}{\:\mathrm{d}}
\newcommand{\dxShort}{\mathrm{d}}
\newcommand{\argument}{\, \cdot \,}
\DeclareMathOperator{\sgn}{sgn}
\DeclareMathOperator{\err}{err}
\newcommand{\one}{\mathbbm{1}}

\theoremstyle{definition}
\newtheorem{rem}{Remark}[section]
\newtheorem{expl}[rem]{Example}

\theoremstyle{plain}
\newtheorem{prop}[rem]{Proposition}

\newtheorem{theo}[rem]{Theorem}
\newtheorem{cor}[rem]{Corollary}

\numberwithin{equation}{section}

\begin{document}

\begin{frontmatter}

\title{A solution to a linear integral equation with an application to statistics of infinitely divisible moving averages.}

\runtitle{An integral equation and statistics of moving averages}
\begin{aug}

\author{\fnms{Jochen} \snm{Gl\"uck}\thanksref{a,e1}\ead[label=e1,mark]{jochen.glueck@alumni.uni-ulm.de}}
\author{\fnms{Stefan} \snm{Roth}\thanksref{a,e2}\ead[label=e2,mark]{stefan.roth@alumni.uni-ulm.de}}
\and
\author{\fnms{Evgeny} \snm{Spodarev}\thanksref{a,e3}\ead[label=e3,mark]{evgeny.spodarev@uni-ulm.de}}
\address[a]{Helmholtzstra\ss e 18, 89075 Ulm. \printead{e1}; \printead{e2}; \printead{e3}}

\runauthor{J. Gl\"uck et al.}

\affiliation{Ulm University}

\end{aug}

\begin{abstract}

For a stationary moving average random field, a non--parametric low frequency estimator of the L\'evy density of its infinitely divisible independently scattered integrator measure is given. The plug--in estimate is based on the solution $w$ of the linear integral equation  $v(x) = \int_{\mathbb{R}^d} g(s) w(h(s)x)ds$, where $g,h:\mathbb{R}^d \rightarrow \R$ are given measurable functions and $v$ is a (weighted) $L^2$-function on $\mathbb{R}$. We investigate conditions for the existence and uniqueness of this solution and give $L^2$--error bounds for the resulting estimates. An application to pure jump moving averages and a simulation study round off the paper.
\end{abstract}


\end{frontmatter}

\section{Introduction} \label{sect:Int}

Consider a stationary infinitely divisible independently scattered random measure $\Lambda$ whose L\'{e}vy characteristics
are given by $(a_0,b_0,v_0)$, where $a_0 \in \R$, $b_0 \geq 0$,  and $v_0$ is a L\'{e}vy density. For some (known) $\Lambda$-integrable
function $f:\R^d \to \R$ let $X = \{ X(t); \ t\in \R^d \}$, with 
\begin{equation}\label{eq:spectral_repr}
X(t) = \int_{\R^d} f(t-x) \Lambda (dx), \quad t \in \R^d,
\end{equation} 
be the corresponding infinitely divisible moving average random field with L\'{e}vy characteristics $(a_1,b_1,v_1)$. 

A wide class of spatio-temporal processes with spectral representation~\eqref{eq:spectral_repr} is provided by the models of turbulent
liquid flows (the so-called ambit random fields, where a space time L\'{e}vy process serves as integrator). Ambit fields cover a lot of different processes and fields including Ornstein-Uhlenbeck type and mixed 
moving average random fields (cf.~\cite{BarnNielsSchmiegel04,BarnNiels11,Podol15} ). Such processes are also used to model
the growth rate of tumours, where the spatial component describes the angle between the center of the tumour cell and the nearest point at its 
boundary (cf.~\cite{BarnNielsSchmiegel07},~\cite{jonsdottir2008}). 
Another interesting application of~\eqref{eq:spectral_repr} is given in~\cite{Karcher12}, where the author uses infinitely divisible moving 
average random fields in order to model claims of natural disaster insurance within different postal code areas.

The interplay between the L\'{e}vy densities $v_0$ and $v_1$ is described by the relation
\begin{equation}\label{eq:v_1_v_0_relation}
	v_1(x) = \int\limits_{\supp (f)} \frac{1}{|f(s)|} v_0 \left( \frac{x}{f(s)} \right) \dx s, \quad x \in \R^\times,
\end{equation}
where $\supp (f)$ denotes the support of $f$, and where $\R^\times = \R \setminus \{ 0 \}$. Given $v_0$ and $f$, the function $v_1$ is determined by this relation. Now, consider the reverse situation: 
assume $f$ and $v_1$ to be known; 
is it possible to recover $v_0$ from this equation? If the answer to this question is positive, i.e., the correspondence $v_0 \leftrightarrow v_1$  in $\eqref{eq:v_1_v_0_relation}$ is one--to--one,    a non--parametric estimator of $v_0$ can be obtained from an estimator $v_1$.

\subsection*{Nonparametric estimation of $v_0$}

Our main results, presented in Section~\ref{sec:stochastic-application}, deal with the nonparametric estimation of $v_0$ from low frequency observations $(X(t_1),\dots,X(t_n))$ 
of the moving average random field $X$. These observations are used to construct an estimator for $v_1$, and we will employ functional analytic results about (a generalization of) the integral equation~\eqref{eq:v_1_v_0_relation}, presented in Section~\ref{sec:integral-equation}, to construct an estimator for $v_0$ from that.

Our results extend those in~\cite{ KarRothSpoWalk19} to the more general case where $f$ is not assumed to be a simple function. The reason for this extension is twofold. First, the approach in \cite{ KarRothSpoWalk19} is numerically difficult to apply if the number of possible values of $f$ is large. Second, our present paper uses a completely different analytic methodology (cf. Section \ref{sec:integral-equation}) which is of value on its own and is general enough to solve other inverse problems of the form \eqref{eq:main}.

Let us comment on the state of the art. {In the case $d=1$, estimation} of the L\'evy density $v_0$ of the integrator L\'evy process $\{ L_s\}$ of a moving average process  
$X(t)=\int_{\R} f(t-s) \, d L_s,$ $t\in\R$, is treated in~\cite{BelPanWoern19}. It is assumed there that $\E\, L^2_0<\infty$. This estimate is based on 
the inversion of the Mellin transform of the second derivative of the cumulant of $X(0)$. A uniform error bound as well as  the consistency of 
the estimate are given. 
However, the logarithmic convergence rate {shown in \cite[Corollary 1]{BelPanWoern19}} is rather slow. In the paper \cite{BelomOrlPan19}, a specific parametric choice of the kernel function $f$ (including the exponential kernel $f(x)=e^{-x}$ as a limiting  case) leads to another Fourier-based approach to estimate the L\'evy triplet of $\Lambda$.

We prefer to construct a plug-in estimator for $v_0$ that is based on estimates for $v_1$ -- or rather on estimates for $uv_1$, where $u:\R^\times \to \R^\times$ is a multiplicative function such as, e.g., $u(x) = \lvert x \rvert^\beta$ for some $\beta \in \R$. Here, the function $u$ occurs since many of the estimators for L\'{e}vy densities are based on derivatives of the Fourier transform; see e.g.\ \cite{neumann,gugushvili,comte}.

The usage of a plug-in estimator for $v_0$ that is constructed from an estimator for $uv_1$ has the following advantages:
\begin{enumerate}[(a)]
	\item Conditions 
	on $f$ and $uv_1$ are simple to check for a larger class of models, compare Corollary \ref{cor:simple-function-sufficient-condition} and Examples \ref{example:corr3_6}, \ref{ex:2.11}, \ref{ex:pure_jump_1}, \ref{ex:pure_jump_2};
	\item Universality: for any $d\geq 1$ and any estimator $\widehat{uv_1}$ of $uv_1$ the $L^2$-approximation error in the estimation of $uv_0$ can be quantified in 
	terms of the input error \\ $\E \left\| \widehat{uv_1} - uv_1 \right\|^2_{L^2(\R^\times, |x|^c \dx x)}$
	(under certain regularity assumptions on $uv_1$);
	\item At least in case that $X$ is a pure jump infinitely divisible moving average random field, one can obtain $L^2$-convergence rates of order 
	$\mathcal{O}(n^{-\gamma})$ for $uv_0$, where $\gamma > 0$ is a model depending constant.  
\end{enumerate}
Note that $L^2$-consistent estimates for $uv_1$ in situation (c) are available e.g. if $X$ is either $\phi$-mixing or $m$-dependent (cf.~\cite{ KarRothSpoWalk19}).
We also mention that in general, mixing conditions are tricky to check and it may be helpful to use relationships between the different notions
of mixing. For further details on this topic see e.g. \cite{Bradley93}, \cite{doukhan}. On the other hand, $X$ is $m$-dependent whenever $f$
has a bounded 
support.

The construction of the estimator in~\cite{ KarRothSpoWalk19} mainly relies on the first derivative of the characteristic function of $X(0)$. Such methods are indeed well-established
for L\'{e}vy processes (cf.~\cite{neumann},~\cite{gugushvili}~and~\cite{comte}). The main difference between L\'{e}vy processes and stationary infinitely 
divisible random fields is the absence of independent increments that makes proofs very hard since techniques for the i.i.d.\ random variables case cannot be applied in
most situations.
\begin{rem}
In this paper, we stick to the case of low--frequency observations $(X(t_1),\dots,X(t_n))$ since (unlike in financial applications for $d=1$) this situation is the most common in modern spatial  data sets ($d\ge 2$). However, our estimation approach for $v_0$ in Section \ref{sec:estimate_v0} is plug--in and does not differ between low and high frequency data. The difference appears first on the level of estimation of the L\'evy density $v_1$ as it is seen in  Section \ref{subsec:pure_jump_random_fields}. There, high--frequency estimators of $v_1$ such as those in \cite{comte} can be used as well which will slightly improve their rates of convergence to $v_1$ in Theorem \ref{eq:err1_bound} (compare  \cite[Corollary 4.1]{comte}) and, hence, the final performance of the estimate of $v_0$, cf. Theorem \ref{theo:error-bound-for-estimator}.
\end{rem}

\subsection*{Solution theory of the underlying integral equation}

Our estimation approach is  based on an effective method to solve the integral equation~\eqref{eq:v_1_v_0_relation} for $v_0$. In fact, assuming an estimator for $uv_1$  to be known,  we  actually solve the equation
\begin{equation}\label{eq:integ_equation}
	v(x) = \int_{\supp (h)} g(s) w(h(s)x) \dx s, \quad x \in \R^\times,
\end{equation}
for $w$, where $v = uv_1$, $w = uv_0$, $h = 1/f$ on the support $\supp(f)$ of $f$ (and $h=0$ outside $\supp(f)$), and $g(s) = \frac{u(f(s))}{\lvert f(s)\rvert}$ for $s \in \supp(f)$ (while $g = 0$ outside $\supp(f)$).

In Section~\ref{sec:integral-equation}, we study the solvability of equation~\eqref{eq:integ_equation} (with general measurable functions $g,h:\R^d \to \R$) for $w$ if $v$  belongs to the weighted $L^2$-space $L^2(\R^\times,|x|^c \dx x)$ for a fixed number $c \in \R$. We seek for a solution $w$ in the same space $L^2(\R^\times,|x|^c \dx x)$. Of course, the purpose of the exponent $c$ is to control the desired integrability of the functions $v$ and $w$. 

In order to prove existence and uniqueness of solutions to~\eqref{eq:integ_equation} one needs to analyze the linear operator
\begin{equation*}
	w \to \Gi w = \int_{\supp (h)} g(s) w(h(s) \argument) \dx s
\end{equation*}
on the Hilbert space $L^2(\R^\times, |x|^c\dx x)$. Notice that the operator $\Gi$ does not fit into classical Fredholm theory. For instance, it is not compact unless for very special choices of $h$ and $g$; this follows from Theorem~\ref{theo:similar-to-multiplication-operator} below.

Instead, we are going to use Fourier analysis on the multiplicative group of real numbers to analyze the operator $\Gi$. We will see that, under mild assumptions on $g$ and $h$,
$\Gi$ is unitarily similar to a 
multiplication operator (cf. Theorem~\ref{theo:similar-to-multiplication-operator}); thus, necessary and sufficient conditions for the
(unique) solvability of $v = \Gi w$ with respect to the unknown function $w$ can be characterized by injectivity and surjectivity properties 
of a multiplicator, which are very simple. In case of existence, we also provide a formula for the solution to~\eqref{eq:integ_equation} (cf. Corollary~\ref{cor:criterion-for-solvability}).

\subsection*{Organization of the paper}

In Section~\ref{sec:stochastic-application}, we construct plug-in estimators for 
the L\'{e}vy density $v_0$ from low frequency observations of the moving average random field $X$. We further provide bounds
for the $L^2$-error in case that $v_0$ satisfies some integrability conditions. Finally, we show how this approach can be applied to pure 
jump infinitely divisible moving average random fields $X$, i.e. when the characteristic function of $X$ has no Gaussian component.

Section~\ref{sec:integral-equation} provides the functional analytic background for the analysis in Section~\ref{sec:stochastic-application}. It starts with a brief overview of Fourier transforms on the multiplicative group
$\R^\times = \R \setminus \{ 0 \}$, followed by the solution theory for the integral equation~\eqref{eq:integ_equation}. Here, we provide necessary and sufficient 
conditions for the existence and uniqueness of a solution.

In Section~\ref{sec:simulation_study}, we demonstrate that our estimation approach works well for simulated data in dimension $d=1,2$.


\section{Estimation of the L\'evy density of infinitely divisible moving average random fields} \label{sec:stochastic-application}

In this section, we introduce a general procedure to construct plug-in estimators for the L\'{e}vy density of the integrator random measure $\Lambda$ in~\eqref{eq:spectral_repr}, provided an estimator for the L\'{e}vy density of $X$ is given.

First we recall a few definitions (Section~\ref{subsec:application-notations}) and give a brief overview of infinitely divisible random measures and fields {(Section~\ref{subsec:reminder})}.

\subsection{Notation} \label{subsec:application-notations}

Let   $(S,\mathcal{S},\mu)$ be a measure space. We write $L^2(S,\mu)$ for the
space of all $\mu$-measurable functions such that $\int_S |f(x)|^2 \mu(dx)$
is finite. If $S \subseteq \R$ is Borel measurable,  let $\mathcal{S} = \B(S)$ be the Borel $\sigma$-field on $S$. In case  that $\mu$ has a density $g = \frac{d\mu}{dx}$ w.r.t.\ the Lebesgue measure on $\B(S)$ we use the notation $L^2(S,g(x)dx)$. Moreover,  we shortly write $L^2(S)$ for $g \equiv 1$.

The usual Fourier transform 
on the additive group $\R$ is subsequently denoted by $\F_+$; it is, up to a multiplicative constant, a unitary operator on $L^2(\R,\dxShort x)$. Similarly, the Fourier transform on the multiplicative group $\R^\times$ is, again up to a multiplicative constant, a unitary operator
\begin{equation*}
\F_\times: L^2\left(\R^\times, \frac{\dxShort x}{\lvert x\rvert}\right) \to L^2\left(\R^\times, \frac{\dxShort x}{\lvert x\rvert}\right),
\end{equation*}
(as $\frac{\dxShort x}{\lvert x\rvert}$ is the Haar measure on the multiplicative group $\R^\times$). Since $\F_\times$ is hardly treated in detail in the literature, we sum up some of its properties in Section~\ref{subsec:fourier-transform-on-multiplicative-real-numbers}.

By $\M$ we denote the unitary operator from $L^2(\R^\times, \lvert x\rvert^c \dx x)$ to $L^2\left(\R^\times, \frac{\dxShort x}{\lvert x\rvert}\right)$ that is given by
\begin{equation*}
(\M u)(x) = \lvert x\rvert^{(c+1)/2} \, u(x)
\end{equation*}
for each $u \in L^2(\R^\times, \lvert x\rvert^c \dx x)$. Finally, 
the Sobolev space of order $\alpha \in [0,\infty)$ on $\R$ will subsequently be denoted by $H^\alpha (\R)$, i.e.
\begin{equation*}
	H^\alpha(\R) = H^\alpha(\R,dx) = \left\{ f\in L^2(\R): \ \int_\R |(\F_+ f)(x)|^2 (1+x^2)^{\alpha} dx < \infty \right\}.
\end{equation*}
For each fuction $f:\R^d\to\R$, let $\supp(f) := \{s \in \R^d: \ f(s) \neq 0\}$ denote the support of $f$.

\subsection{Infinitely divisible random measures -- a brief reminder} \label{subsec:reminder}

Throughout, we denote the Borel $\sigma$-field on the $d$-dimensional Euclidean space $\R^d$ by $\mathcal{B}(\R^d)$. The collection of all bounded Borel sets in $\R^d$ will be denoted by $\mathcal{E}_0(\R^d)$. 

Let $\Lambda = \{\Lambda(A); \ A \in \mathcal{E}_0(\R^d)\}$ be an infinitely divisible random measure on some probability space $(\Omega, \mathcal{A}, P)$, i.e.\ a random measure with the following properties:
\begin{enumerate}
	\item[(a)] Let $(E_m)_{m\in \N}$ be a sequence of disjoint sets in $\mathcal{E}_0(\R^d)$. Then the sequence $(\Lambda(E_m))_{m\in \N}$ consists of independent random variables; if, in addition, $\cup_{m=1}^\infty E_m \in \mathcal{E}_0(\R^d)$, then we have $\Lambda(\cup_{m=1}^\infty E_m) =\sum_{m=1}^\infty \Lambda(E_m)$ almost surely.
	\item[(b)] The random variable $\Lambda(A)$ has an infinitely divisible distribution for any choice of $A \in \mathcal{E}_0(\R^d)$.  
\end{enumerate}

For every $A\in \mathcal{E}_0(\R^d)$, let $\varphi_{\Lambda(A)}$ denote the characteristic function of the random variable $\Lambda(A)$. Due to the infinite divisibility of the random variable $\Lambda(A)$, the characteristic function $\varphi_{\Lambda(A)}$ has a L\'{e}vy-Khintchin representation which can, in its most general form, be found in \cite[p.\ 456]{Rajput}. Throughout the rest of the paper we make the additional assumption that the L\'{e}vy-Khintchin representation of $\Lambda(A)$ is of a special form, namely
\begin{equation*} 
	\varphi_{\Lambda(A)}(t) = \exp \left\lbrace \nu_d(A) K(t) \right\rbrace, \quad A \in \mathcal{E}_0(\R^d), 
\end{equation*}
with
\begin{equation}\label{eq:K} 
K(t) = ita_0 - \frac{1}{2} t^2 b_0 + \int\limits_{\R} \left( e^{itx} - 1 - itx \id_{[-1,1]}(x) \right)v_0(x)dx,
\end{equation}
where $\nu_d$ denotes the Lebesgue measure on $\R^d$, $a_0$ and $b_0$ are real numbers with $0 \leq b_0 < \infty$ and $v_0: \R \to \R$ is a L\'{e}vy density, i.e.\ a measurable function which satisfies $\int_{\R} \min \{1,x^2\}v_0(x)dx < \infty$. The triplet $(a_0,b_0,v_0)$ will be referred to as {\it L\'{e}vy characteristic} of $\Lambda$. It uniquely determines the distribution of $\Lambda$.
This particular structure of the characteristic functions $\varphi_{\Lambda(A)}$ means that the random measure $\Lambda$ is stationary with control measure $\lambda: \mathcal{B}(\mathbb{\R}) \rightarrow [0,\infty)$ given by  
\begin{equation*}
	\lambda(A) = \nu_d(A) \left[  |a_0| + b_0 + \int\limits_{\R} \min \{1,x^2\} v_0(x)dx \right] \qquad \text{for all } A \in \mathcal{E}_0(\R^d).
\end{equation*}

Now one can define the stochastic integral with respect to the infinitely divisible random measure $\Lambda$ in the following way: 
\begin{enumerate}
	\item Let $f = \sum_{j=1}^n x_j \id_{A_j}$ be a real simple function on $\R^d$, where $A_j \in \mathcal{E}_0(\R^d)$ are pairwise disjoint. Then for every $A \in \mathcal{B}(\R^d)$ we define
	\begin{equation*}
		\int\limits_{A}f(x)\Lambda(dx) = \sum\limits_{j=1}^n x_j \Lambda(A \cap A_j).
	\end{equation*}
	\item A measurable function $f:(\R^d,\mathcal{B}(\R^d))\rightarrow (\R, \mathcal{B}(\R))$ is said to be $\Lambda$-integrable if there exists a sequence $(f^{(m)})_{m \in \mathbb{N}}$ of simple functions as in (1) such that $f^{(m)} \rightarrow f$ holds $\lambda$-almost everywhere and such that, for each $A \in \mathcal{B}(\R^d)$, the sequence $\left( \int_{A} f^{(m)}(x)\Lambda(dx) \right)_{m \in \mathbb{N}}$ converges in probability as $m \rightarrow \infty$. In this case we set
	\begin{equation*}
		\int\limits_{A} f(x) \Lambda(dx) = \plim\limits_{m\rightarrow \infty} \int\limits_{A} f^{(m)}(x)\Lambda(dx).
	\end{equation*}
\end{enumerate}

A useful characterization for $\Lambda$-integrability of a function $f$ is given in \cite[Theorem 2.7]{Rajput}. Now let $f: \R^d \rightarrow \R$ be $\Lambda$-integrable; then the function $f(t - \cdot)$ is $\Lambda$-integrable for every $t\in \R^d$ as well. We define the moving average random field $X = \{X(t), \ t \in \R^d\}$ by
\begin{equation*}
	X(t) = \int\limits_{\R^d} f(t-x)\Lambda(dx), \quad t \in \R^d.
\end{equation*}
Recall that a random field is called \emph{infinitely divisible} if its finite dimensional distributions are infinitely divisible. The random field $X$ above is (strictly) stationary and infinitely divisible. The characteristic function $\varphi_{X(0)}$ of $X(0)$ is given by
\begin{equation*}
	\varphi_{X(0)}(u) = \exp \left( \int_{\R^d} K(uf(s)) \dx s \right), \quad u \in \R,
\end{equation*}
where $K$ is the function from~\eqref{eq:K}. The argument $\int_{\R^d} K(uf(s)) \dx s$ in the above exponential function can be shown to have a similar structure as $K(t)$; more precisely, we have
\begin{equation*}
	\int_{\R^d} K(uf(s)) \dx s = i u a_1 - \frac{1}{2}u^2 b_1 + \int\limits_{\R} \left(e^{iux}-1-iux\id_{[-1,1]}(x)\right)v_1(x) \dx x 
\end{equation*} 
where $a_1$ and $b_1$ are real numbers with $b_0 \geq 0$ and the function $v_1$ is the L\'{e}vy density of $X(0)$. The triplet $(a_1,b_1,v_1)$ is again referred to as \emph{L\'{e}vy characteristic} (of $X(0)$) and determines the distribution of $X(0)$ uniquely. A simple computation shows that the triplet $(a_1, b_1, v_1)$ is given by the formulas
\begin{align} \label{eq:levy-characterisitic-of-field}
	\begin{split}
		& a_1 = \int\limits_{\mathbb{R}^d}U(f(s)) \dx s, \qquad  b_1 = b_0 \int\limits_{\R^d} f^2(s) \dx s, \\
		& v_1(x) = \int\limits_{\supp(f)}  \frac{1}{|f(s)|}v_0\left( \frac{x}{f(s)} \right) \dx s,
	\end{split}
\end{align} 
where  the function $U$ is defined via
\begin{equation*}
	U(u) = u \left( a_0 + \int_{\mathbb{R}} x \left[ \id_{[-1,1]}(ux)- \id_{[-1,1]}(x) \right] v_0(x)\dx x \right).
\end{equation*}
Note that $\Lambda$-integrability of $f$ immediately implies that $f \in L^1(\R^d) \cap L^2(\R^d)$. Hence, all integrals above are finite.

For details on the theory of infinitely divisible measures and fields we refer the interested reader to \cite{Rajput}.


\subsection{An inverse problem}  \label{subsection:inverse-problem}

Throughout the rest of Section~\ref{sec:stochastic-application}, let 
the random measure $\Lambda = \{\Lambda(A), \ A \in \mathcal{E}_0(\R^d) \}$, 
the function $f: \R^d \rightarrow \R$ and the random field $X$ be given as in Section~\ref{subsec:reminder}. Moreover, 
we fix an exponent $c \geq 0$ 
for the weight of the measure in the Hilbert space $L^2(\R^\times, \lvert x\rvert^c \dx x)$.

In typical applications, one can observe the random field $X$ and, from those observations, compute an estimator $(\hat a_1, \hat b_1, \hat v_1)$ for the L\'{e}vy characteristic $(a_1, b_1, v_1)$ of $X(0)$. If one is interested in the random measure $\Lambda$ one needs to compute an estimator for the L\'{e}vy characteristic $(a_0, b_0, v_0)$, given only the estimator for $(a_1, b_1, v_1)$. Those two triplets are related by the formulas~\eqref{eq:levy-characterisitic-of-field}. Assuming that $f$ is known, those formulas immediately yield a way to compute an estimator $\hat b_0$ from the estimator $\hat b_1$. In order to compute an estimator $\hat v_0$ from the estimator $\hat v_1$ one needs to solve  the integral equation \eqref{eq:v_1_v_0_relation}.
	

Once this is accomplished, it is also not difficult to derive an estimator 
$\hat a_0$ from relations~\eqref{eq:levy-characterisitic-of-field} provided that $\int_{\R} f(s) ds \neq 0$ (cf.~\cite{Karcher12, KarRothSpoWalk19}).  
Hence, the main difficulty is to solve the equation
\eqref{eq:v_1_v_0_relation}
for $v_0$ if $v_1$ is given. Our results from Section~\ref{sec:integral-equation} can be used to discuss whether this equation has a unique solution in $L^2(\R^\times, \lvert x\rvert^c \dx x)$ for given $v_1 \in L^2(\R^\times, \lvert x\rvert^c \dx x)$; they also show us how the solution, provided that it exists, can be computed by using only multiplication operators and Fourier transforms.

Now, it turns out that things are actually a bit more involved than discussed above, for the following reason: given a list of observations of the random field $X$ it is common to estimate the function $x^\beta v_1(x)$ rather than $v_1(x)$ itself, since many of the estimators for L\'{e}vy densities are based on derivatives of the Fourier transform (over the additive group $\R$); see e.g.\ \cite{neumann,gugushvili,comte} where this can be seen in the context of L\'{e}vy processes.

To put this in a more general setting, fix $\beta \in \R$ and let $u: \R^\times \to \R$ be a function which is either given by $u(x) = \lvert x\rvert^\beta$ for all $x \in \R^\times$ or by $u(x) = \sgn(x) \lvert x\rvert^\beta$ for all $x \in \R^\times$. Then $u$ is Borel measurable and multiplicative, i.e.\ we have $u(xy) = u(x)u(y)$ for all $x,y \in \R^\times$. Assuming that the function $uv_1$ is in $L^2(\R^\times, \lvert x\rvert^c \dx x)$ we would like to compute the function $uv_0$ in case that this function is still contained in $L^2(\R^\times, \lvert x\rvert^c \dx x)$. Using the relation of $v_0$ and $v_1$ and the fact that $u$ is multiplicative, we immediately obtain the equation
\begin{equation}\label{eq:main-stochatic-part}
(u v_1)(x) = \int\limits_{\supp(f)} \frac{u(f(s))}{|f(s)|} (u v_0)\left( \frac{x}{f(s)} \right) \dx s.
\end{equation}

This is an integral equation of the type~\eqref{eq:main} that we study in detail in Section~\ref{sec:integral-equation} (where $h = \frac{1}{f} \one_{\supp(f)}$ and $g(s) = \frac{u(f(s))}{|f(s)|}$ for $s \in \supp(f)$). In order to apply our general functional analytic results from from Section~\ref{sec:integral-equation} -- in particular, Theorem~\ref{theo:similar-to-multiplication-operator} and its corollaries -- we need that a specific integrability condition \eqref{eq:integrability-condition} is satisfied. To this end, we have to assume that
\begin{equation}
\label{eq:integrability-condition-stochastic-part}
\int\limits_{\supp(f)} \lvert f(s)\rvert^{\beta + \frac{c-1}{2}} \dx s < \infty.
\end{equation}
This condition is not too restrictive and is satisfied for many functions $f$ with a specific choice of constants $\beta$ and $c$ used in applications (compare Examples  \ref{example:corr3_6}, \ref{ex:2.11}, \ref{ex:pure_jump_1}  and Corollary \ref{cor:simple-function-sufficient-condition}).
Then we can define functions $m_{f,\pm}: \R \to \C$ and $\mu_f: \R^\times \to \C$ by  
\begin{equation}
\label{eq:functions-m_f}
\begin{split}
m_{f,+}(x) & := \int\limits_{\supp(f)} u(f(s)) \, \lvert f(s)\rvert^{(c-1)/2} \, e^{-i x \, \log \lvert f(s)\rvert} \dx s, \\
m_{f,-}(x) & := \int\limits_{\supp(f)} u(f(s)) \, \lvert f(s)\rvert^{(c-1)/2} \, e^{-i x \, \log \lvert f(s)\rvert} \, \sgn f(s) \dx s, \\
\mu_f(y) & :=
\begin{cases}
m_{f,+}(\log\lvert y\rvert) \quad & \text{if } y > 0, \\
m_{f,-}(\log\lvert y\rvert) \quad & \text{if } y < 0.
\end{cases}
\end{split}
\end{equation}

The functions $m_{f,\pm}$ and $\mu_f$ are bounded and continuous. They play a major role in the following Theorem~\ref{theo:solvability-of-stochastic-integral-equation} 
which immediately yields criteria for the existence and uniqueness of solutions to equation~\eqref{eq:main-stochatic-part}. It is a consequence of more general Theorem~\ref{theo:similar-to-multiplication-operator} and its Corollaries~\ref{cor:injectivit-and-surjectivity-of-G} and~\ref{cor:criterion-for-solvability}.

\begin{theo} \label{theo:solvability-of-stochastic-integral-equation}
	Assume that the integrability condition~\eqref{eq:integrability-condition-stochastic-part} is satisfied.
	\begin{enumerate}
		\item[(a)] The solution $uv_0 \in L^2(\R^\times, \lvert x\rvert^c \dx x)$ of the integral equation~\eqref{eq:main-stochatic-part} is unique for all $uv_1 \in L^2(\R^\times, \lvert x\rvert^c \dx x)$ for which it exists iff $m_{f,+} \not= 0$ and $m_{f,-} \not= 0$ almost everywhere on $\R$.
		\item[(b)] The integral equation~\eqref{eq:main-stochatic-part} has a solution $uv_0 \in L^2(\R^\times, \lvert x\rvert^c \dx x)$ for all $uv_1 \in L^2(\R^\times, \lvert x\rvert^c \dx x)$ iff $\inf_{x \in \R} \lvert m_{f,+}(x)\rvert > 0$ and $\inf_{x \in \R} \lvert m_{f,-}(x)\rvert > 0$.
		\item[(c)] Let $\alpha \ge 0$ and assume that, for all $x \in \R$ and a constant $\gamma > 0$,
		\begin{equation*}
			\lvert m_{f,\pm}(x)\rvert \ge \frac{\gamma}{1 + \lvert x\rvert^\alpha}.
		\end{equation*}
		If $uv_1 \in L^2(\R^\times, \lvert x\rvert^c \dx x)$ and if both the functions
		\begin{align*}
			& x \mapsto |x|^{(c+1)/2}(uv_1)(\exp(x)) \\
			\text{and} \quad & x \mapsto |x|^{(c+1)/2}(uv_1)(-\exp(x))
		\end{align*}
		are elements of the Sobolev space ${H^{\alpha}(\R)}$,
		then the equation~\eqref{eq:main-stochatic-part} has a unique solution
		$uv_0 \in L^2(\R^\times, \lvert x\rvert^c \dx x)$.
	\end{enumerate}
\end{theo}

Note that, while our proof of Theorem~\ref{theo:similar-to-multiplication-operator} is formulated over the complex field (as it employs Fourier analysis), it does not matter in the statement of Theorem~\ref{theo:solvability-of-stochastic-integral-equation} (nor in the statement of Theorem~\ref{theo:similar-to-multiplication-operator}) whether we consider complex-valued or real-valued functions. This is explained in detail in Remark~\ref{rem:real-vs-complex}.

Let us briefly discuss what Theorem~\ref{theo:solvability-of-stochastic-integral-equation} gives us in case that $f$ is a simple function: let $f = \sum_{j=1}^n f_j \one_{\Delta_j}$, where $f_1,\dots,f_n \in \R \setminus \{0\}$ are pairwise distinct numbers and where $\Delta_1,\dots,\Delta_n \in \B(\R)$ are pairwise disjoint sets of finite Lebesgue measure. Then the integrability condition~\eqref{eq:integrability-condition-stochastic-part} is automatically satisfied and the functions $m_{f,+}$ and $m_{f,-}$ take the form
\begin{align} \label{eq:functions-m-for-simple-f}
	\begin{split}
		m_{f,+}(x) & = \sum_{j=1}^n u(f_j) \lvert f_j\rvert^{(c-1)/2} e^{-ix \log \lvert f_j\rvert} \nu_d(\Delta_j), \\
		m_{f,-}(x) & = \sum_{j=1}^n u(f_j) \lvert f_j\rvert^{(c-1)/2} e^{-ix \log \lvert f_j\rvert} \nu_d(\Delta_j) \sgn f_j.
	\end{split}
\end{align}
Hence, Theorem~\ref{theo:solvability-of-stochastic-integral-equation} yields the following corollary which mimics~\cite[Theorem 4.1]{KarRothSpoWalk19}:

\begin{cor}\label{cor:simple-function-sufficient-condition}
	Let $f$ be a simple function as above. Suppose that 
	\begin{equation}\label{eq:if_condition_simple_f}
	\sum_{j=2}^n \left(\frac{\lvert f_j\rvert }{\lvert f_1\rvert} \right)^{\beta + \frac{c-1}{2}} \frac{\nu_d(\Delta_j)}{\nu_d(\Delta_1)} < 1.
	\end{equation}
	Then the integral equation~\eqref{eq:main-stochatic-part} has a unique solution $uv_0 \in L^2(\R^\times, \lvert x\rvert^c \dx x)$ for all $uv_1 \in L^2(\R^\times, \lvert x\rvert^c \dx x)$.
\end{cor}
\begin{proof}
	The inequality~\eqref{eq:if_condition_simple_f} is equivalent to 
	\begin{equation*}
		\sum_{j=2}^n \lvert f_j\rvert^{\beta + \frac{c-1}{2}} \nu_d(\Delta_j) < \lvert f_1 \rvert^{\beta + \frac{c-1}{2}} \nu_d(\Delta_1);
	\end{equation*}
	a glance at the formulas for $m_{f,+}$ and $m_{f,-}$ in~\eqref{eq:functions-m-for-simple-f} shows that this implies, by means of the triangle inequality, that $\inf_{x \in \R} \lvert m_{f,\pm}(x)\rvert > 0$.
\end{proof}

The formulas for $m_{f,\pm}$ given in~\eqref{eq:functions-m-for-simple-f} in case that $f$ is a simple function give us the opportunity to construct counterexamples to several naturally arising questions:

\begin{expl}
	The fact that $m_{f,\pm} \not= 0$ almost everywhere on $\R$ does not imply that $\inf_{x} \lvert m_{f,\pm}(x) \rvert > 0$ (i.e.\ uniqueness of solutions for the integral equation~\eqref{eq:integrability-condition-stochastic-part} does not imply existence).
	
	Indeed, set $c = 1$, $\beta = 0$ and $u \equiv 1$ for the sake of simplicity. Let $\Delta_1, \Delta_2 \in \B(\R^d)$ be disjoint sets of Lebesgue measure $1$ and define $f = \one_{\Delta_1} + e \one_{\Delta_2}$. The formulas~\eqref{eq:functions-m-for-simple-f} then yield
	\begin{equation*}
		m_{f,+}(x) = m_{f,-}(x) = 1 + e^{-ix}
	\end{equation*}
	for all $x \in \R$. This function has countable many zeros on the real line, so $m_{f,\pm} \not= 0$ almost everywhere but $\inf_{x \in \R} \lvert m_{f,\pm}(x)\rvert = 0$.
\end{expl}

\begin{expl}
	The inequality~\eqref{eq:if_condition_simple_f} is only a sufficient, but not a necessary condition for the conclusion of the corollary.
	
	Indeed, let $c \in \R$ be arbitrary, let $\beta \neq -(c-1)/2$ and $u(x) = |x|^\beta$. Let $\Delta_1, \Delta_2$ and $\Delta_3$ be three disjoint Borel sets of Lebesgue measure $1$. We define $f = \sum_{k=1}^3 e^{\alpha k}\one_{\Delta_k}$ for a fixed number $\alpha \in \R$ that satisfies the estimates
	\begin{equation}\label{eq:alpha_choice}
	\begin{cases}
	 \frac{ \log \left( \frac{-1 + \sqrt{5}}{2} \right) }{\beta + \frac{c-1}{2}} \leq \alpha \leq  
	\frac{\log \left( \frac{1 + \sqrt{5}}{2} \right)}{\beta + \frac{c-1}{2}} , & \text{if }  \beta > -\frac{c-1}{2}, \\
	\frac{\log \left( \frac{1 + \sqrt{5}}{2} \right)}{\beta + \frac{c-1}{2}}   \leq \alpha \leq \frac{
	\log \left( \frac{-1 + \sqrt{5}}{2} \right)}{\beta + \frac{c-1}{2}} , &  \text{if }  \beta < -\frac{c-1}{2}.
	\end{cases}
	\end{equation}
	Then the inequality~\eqref{eq:if_condition_simple_f} is not fulfilled, no matter how we permute the indices $1$, $2$ and $3$. Indeed, with the notation $z = e^{\alpha(\beta + (c-1)/2)} > 0$ we obtain
	\begin{equation*}
		\sum_{j=2}^3 \left( \frac{ f_j }{ f_1} \right)^{\beta + \frac{c-1}{2}} = \begin{cases}
			z + z^2, & f_1 = e^{\alpha}, \\
			z^{-1} + z, & f_1 = e^{2\alpha}, \\
			z^{-2} + z^{-1}, & f_1 = e^{3\alpha}.
		\end{cases}
	\end{equation*}
	The second of the previous terms exceeds $1$ for any $z > 0$ whereas the first and the last terms are greater or equal to $1$ at the same time if and only if 
	$(-1+\sqrt{5})/2 \leq z \leq (1+\sqrt{5})/2$. Resubsituting $z$ shows that this is equivalent to~\eqref{eq:alpha_choice}. Moreover,
	we obtain from formula~\eqref{eq:functions-m-for-simple-f} that
	\begin{equation*}
		m_{f,+}(x) = m_{f,-}(x) = \sum_{k=1}^3  e^{\alpha (\beta + (c-1)/2)k} e^{-i \alpha k x}
	\end{equation*}
	for all $x \in \R$. A simple calculation yields that the complex polynomial $p(w) =\sum_{k=1}^3 e^{\alpha k (\beta+\frac{c-1}{2})} w^k$ has no zeros on the
	complex unit circle $\{ w \in \C: \ |w|= 1 \}$. Thus, $m_{f,+}(x) = p(e^{-i \alpha x}) \neq 0$ for all $x \in \R$. Since $p$ has at most three different zeros in $\C$, we conclude 
	from the continuity of $x \to p(e^{-i \alpha x})$ that $\inf_{x \in \R} |m_{f,+}(x)| > 0$. Hence, Theorem~\ref{theo:solvability-of-stochastic-integral-equation}(b) shows that our integral equation~\textcolor{blue}{\eqref{eq:main-stochatic-part}} has a unique solution $uv_0 \in L^2(\R^\times, \lvert x\rvert^c \dx x)$ for every $uv_1 \in L^2(\R^\times, \lvert x\rvert^c \dx x)$.
\end{expl}

\subsection{An estimator for the L\'{e}vy density $v_0$}\label{sec:estimate_v0}

We use the same notation as in Section~\ref{subsection:inverse-problem}. 
Assume that the integrability condition~\eqref{eq:integrability-condition-stochastic-part} is satisfied and, given the function $uv_1 \in L^2(\R^\times, \lvert x\rvert^c\dx x)$, we compute the function $uv_0 \in L^2(\R^\times,\lvert x\rvert^c\dx x)$. The relation between the functions $uv_1$ and $uv_0$ is given by the integral equation~\eqref{eq:main-stochatic-part} which can, for short, be written as
$uv_1 = \Gi(uv_0)$;
here, $\Gi$ is given by
\begin{align*}
	\Gi: \; L^2(\R^\times,\lvert x\rvert^c\dx x) & \to L^2(\R^\times,\lvert x\rvert^c\dx x), \\
	w & \mapsto \int\limits_{\supp(h)} \frac{u(f(s))}{|f(s)|} \; w\big(\frac{1}{f(s)} \argument \big) \dx s,
\end{align*}
which is a bounded linear operator by condition~\eqref{eq:integrability-condition-stochastic-part}, compare Proposition~\ref{prop:operator-on-L_2}.

In applications we are only given an estimator $\widehat{uv_1}$ for $uv_1$ which depends on the sample size $n \in \N$. Our {general} solution theory from Section~\ref{subsection:solution-theory} {below} provides us with a way to compute the inverse operator $\Gi^{-1}$ by means of the formula
$$
\Gi^{-1} = \M^{-1}\F_\times^{-1} \big(\frac{1}{\mu_f} \F_\times\M \argument\big)
$$
(provided that $\mu_f\not= 0$ almost everywhere); here we 
recall from Section~\ref{subsec:application-notations} that $\M$ is a Hilbert space isomorphism from $L^2(\R^\times, \lvert x\rvert^c \dx x)$ to $L^2(\R^\times, \frac{\dxShort x}{\lvert x\rvert})$ given by
\begin{equation*}
(\M u)(x) = \lvert x\rvert^{(c+1)/2} \, u(x)
\end{equation*}
for each $u \in L^2(\R^\times, \lvert x\rvert^c \dx x)$ and that $\F_\times: L^2(\R^\times, \frac{\dxShort x}{\lvert x\rvert}) \to L^2(\R^\times, \frac{\dxShort x}{\lvert x\rvert})$
is the Fourier transform on the multiplicative group $\R^\times$ (which is explained in detail in Section~\ref{subsec:fourier-transform-on-multiplicative-real-numbers}).

\begin{rem}
Surjectivity of the operator $\Gi$ from Theorem \ref{theo:solvability-of-stochastic-integral-equation} (b) is not an issue in statistics since the estimator of $uv_1$  based on a finite sample from the field $X$ is usually locally integrable (cf. Section \ref{subsec:pure_jump_random_fields}), and all integral operators in the composition of $\Gi^{-1}$ are numerically approximated by integral sums on bounded domains. On the contrary, the injectivity of $\Gi$ (compare Theorem \ref{theo:solvability-of-stochastic-integral-equation} (a)) is important to ensure that the inverse problem is not ill--posed. If $\Gi$ were not injective, several possible candidates for $uv_0$ would exist, and one would either have to describe them all or justify why our inversion approach picks up one of them.
\end{rem}

{It might thus} seem quite natural to define an estimator $\widehat{uv_0}$ for $uv_0$ by means of the formula $\widehat{uv_0} = \Gi^{-1} \widehat{uv_1}$. Unfortunately, this approach does not work for the following main reason. 
  Even if  $\widehat{uv_1}$ lies within the range of $\Gi$ the inverse $\Gi^{-1}$ is in general not  a continuous operator, so we cannot expect $\Gi^{-1} \widehat{uv_0}$ to converge to $uv_0$ as the sample size $n$ tends to infinity. The point here is that the function $\frac{1}{\mu_f}$ is not necessarily bounded, and hence multiplication by this function does not define a bounded linear operator on $L^2(\R^\times, \frac{\dxShort x}{\lvert x\rvert})$.

We can address this challenge as follows: assume that the function $\mu_f: \R^\times \to \C$ defined in formula~\eqref{eq:functions-m_f} satsifies $\mu_f \not= 0$ almost everywhere. 
Choose an arbitrary sequence $(a_n)_{n \in \N} \subseteq (0,\infty)$ which converges to $0$ as $n \to \infty$. For each $n \in \N$ we use the notation
\begin{equation*}
	\frac{1}{\mu_{f,n}} := \frac{1}{\mu_f} \one_{\{\lvert \mu_f\rvert > a_n\}}.
\end{equation*}
Note that the function $\frac{1}{\mu_{f,n}}$ converges almost everywhere to $\frac{1}{\mu_f}$ as $n \to \infty$. Now we can finally define an estimator $\widehat{uv_0}$ by means of the formula
\begin{equation} \label{eq:our-estimator}
\widehat{uv_0} := \M^{-1}\F_\times^{-1} \big( \frac{1}{\mu_{f,n}} \F_\times \M \widehat{uv_1} \big) \qquad \text{for all } n \in \N;
\end{equation}
we point out that the multiplication operator on $L^2(\R^\times, \frac{\dxShort x}{\lvert x\rvert})$ in this formula is continuous (and everywhere defined) since the function $\frac{1}{\mu_{f,n}}$ is bounded.

\begin{rem}
	A few remarks on the numerical evaluation of formula~\eqref{eq:our-estimator} are in order. Note that the operators $\M$ and $\M^{-1}$ in~\eqref{eq:our-estimator} are simply multiplications with real-valued functions, so they are very easy to evaluate numerically; a similar comment applies for the multiplication with $\frac{1}{\mu_{f,n}}$.	
	
		To evaluate the operator $\F_\times$ (given by the integral representation~\eqref{eq:fourier-transform}) one needs a good numerical implementation of the Fourier transform. If it is available, the inverse transfrom $\F_\times^{-1}$ can also be easily computed since we have
	\begin{align*}
		\F_\times^{-1}\varphi = \frac{1}{4\pi}\F_\times \psi
	\end{align*}
	for $\varphi \in L^2(\R^\times, \frac{\dxShort x}{\lvert x\rvert})$, where $\psi \in L^2(\R^\times, \frac{\dxShort x}{\lvert x\rvert})$ is given by $\psi(x) = \varphi(1/x)$ for $x \in \R^\times$
	(see formula~\eqref{eq:inverse-multiplicative-fourier-transform} in Section~\ref{subsec:fourier-transform-on-multiplicative-real-numbers}).
	
\end{rem}

Now we are going to show that the estimator $\widehat{uv_0}$ converges to $uv_0$ as $n \to \infty$, provided that the null sequence $(a_n)$ is appropriately chosen.

\begin{theo} \label{theo:error-bound-for-estimator}
	Let $\err_0^2(n) := \E \big\lVert  \widehat{uv_0} - uv_0 \big\rVert_{L^2(\R^\times, \lvert x \rvert^c \dx x)}^2$, $\err_1^2(n) := \E \big\lVert  \widehat{uv_1} - uv_1 \big\rVert_{L^2(\R^\times, \lvert x \rvert^c \dx x)}^2$ denote the mean square errors of the estimators $\widehat{uv_0}$, $\widehat{uv_1}$, respectively. Then 
	\begin{align*}
		\err_0(n) \; \le \; \frac{1}{a_n} \err_1(n) + \frac{1}{2\sqrt{\pi}} \left\lVert \one_{\{ \lvert \mu_f\rvert \le a_n \} } \frac{1}{\mu_f} \F_\times\M uv_1 \right\rVert_{L^2(\R^\times,\frac{\dxShort x}{\lvert x\rvert})}
	\end{align*}
	for each $n \in \N$. In particular, the estimator $\widehat{uv_0}$ for $uv_0$ is consistent 
	in quadratic mean, provided that $\widehat{uv_1}$ for $uv_1$ is consistent in quadratic mean and that the convergence of $a_n$ to $0$ is sufficiently slow.
\end{theo}

We point out that $\err_1^2(n)$ in the above error bound can be controlled for certain choices of the estimator $\widehat{uv_1}$ (cf.~\cite{ KarRothSpoWalk19}), which will briefly be discussed 
in Section~\ref{subsec:pure_jump_random_fields}.

\begin{proof}[Proof of Theorem~\ref{theo:error-bound-for-estimator}]
	We have
	\begin{align*}
		\err_0(n) \le & \left( \E \left\lVert \M^{-1}\F_\times^{-1} \big( \frac{1}{\mu_{f,n}} \F_\times \M (\widehat{uv_1} - uv_1) \big) \right\rVert_{L^2(\R^\times, \lvert x \rvert^c \dx x)}^2 \right)^{1/2} \\
		& + \left\lVert \M^{-1} \F_\times^{-1}\left( \big(\frac{1}{\mu_{f,n}} - \frac{1}{\mu_f}\big) \F_\times\M uv_1 \right) \right\rVert_{L^2(\R^\times, \lvert x \rvert^c \dx x)} \\
		\le & \frac{1}{2\sqrt{\pi}} \left( \E \left\lVert \frac{1}{\mu_{f,n}} \F_\times \M (\widehat{uv_1} - uv_1) \right\rVert_{L^2(\R^\times, \frac{\dxShort x}{\lvert x \rvert})}^2 \right)^{1/2} \\
		& + \frac{1}{2\sqrt{\pi}} \left\lVert \big(\frac{1}{\mu_{f,n}} - \frac{1}{\mu_f}\big) \F_\times\M uv_1 \right\rVert_{L^2(\R^\times, \frac{\dxShort x}{\lvert x \rvert})} \\
		\le & \frac{1}{a_n} \err_1(n) + \frac{1}{2\sqrt{\pi}} \left\lVert \one_{\{ \lvert \mu_f\rvert \le a_n \}} \frac{1}{\mu_f} \F_\times\M uv_1 \right\rVert_{L^2(\R^\times,\frac{\dxShort x}{\lvert x\rvert})}.
	\end{align*}
	Since $\frac{1}{\mu_f} \F_\times\M uv_1$ is an element of $L^2(\R^\times,\frac{\dxShort x}{\lvert x\rvert})$ and since $\one_{\{ \lvert \mu_f\rvert \le a_n \}}$ converges to $0$ almost everywhere as $n \to \infty$, we conclude from the dominated convergence theorem that the second summand in the estimate tends to $0$. If the estimator $\widehat{uv_1}$ is consistent in quadratic mean and if the convergence of the sequence $a_n$ to $0$ is sufficiently slow (such that $\frac{1}{a_n}\err_1(n) \to 0$), this implies that $\err_0(n) \to 0$ as $n \to \infty$.
\end{proof}

\begin{rem}
	Note that $\frac{\widehat{uv_0}}{u}$ is not a L\'{e}vy density in general, since it cannot be guaranteed to be
	nonnegative. For this purpose, we provide the alternative estimator $\widetilde{uv_0}$ for $uv_0$ defined by
	\begin{equation}\label{eq:alternative_estimator}
	\widetilde{uv_0}(x) := \begin{cases}
	\widehat{uv_0}(x), & \text{if} \ \frac{\widehat{uv_0}(x)}{u(x)} \geq 0, \\
	0, & \text{otherwise}.
	\end{cases}
	\end{equation}
	It is immediately clear that $\widetilde{uv_0} \in L^2(\R^\times, \lvert x \rvert^c \dx x)$ and  
	$$\widetilde{\err}_0^2(n) := \E \big\lVert  \widetilde{uv_0} - uv_0 \big\rVert_{L^2(\R^\times, \lvert x \rvert^c \dx x)}^2 \leq \err_0^2(n).$$
	Hence, the upper bound for $\err_0(n)$ given in Theorem~\ref{theo:error-bound-for-estimator} is an upper bound for $\widetilde{\err}_0(n)$
	as well.
\end{rem}

Finally, we derive a bound for $\err_0(n)$ (and thus for $\widetilde{\err}_0(n)$ as well) provided that $u v_1$ and the function $\mu_f$ satisfy similar 
conditions as  in Corollary~\ref{cor:criterion-for-solvability}:

\begin{cor} \label{cor:criterion-for-convergence}
	Assume that functions $m_{f,+},m_{f,-}: \R \to \C$ defined in~\eqref{eq:functions-m_f} satisfy the estimate
	\begin{equation}\label{eq:m_f_subbound_condition}
	\lvert m_{f,\pm}(x)\rvert \ge \frac{\gamma}{1 + \lvert x\rvert^{\alpha_1}}
	\end{equation}
	for all $x \in \R$, an exponent $\alpha_1 > 0$ and a constant $\gamma > 0$. If 
	\begin{equation*}
		(\M uv_1)(\exp(\argument)), \ (\M uv_1)(-\exp(\argument)) \in H^{\alpha_2}(\R^\times, \dx x)
	\end{equation*}
	for a real number $\alpha_2 > \alpha_1$, then we have
	\begin{align*}
		\err_0(n) \le \frac{\err_1(n)}{a_n} + C \, a_n^{\frac{\alpha_2}{\alpha_1} - 1}
	\end{align*}
	for all $n \in \N$ and a constant $C \ge 0$.
\end{cor}
\begin{proof}
	Using the assumption~\eqref{eq:m_f_subbound_condition}, we can find a number $D \ge 0$ such that 
	\begin{equation*}
		\frac{1}{\lvert \mu_f(x)\rvert^{\alpha_2/\alpha_1}} \le D\big(1 + \left\lvert\log\lvert x\rvert \right\rvert^{\alpha_2}\big)
	\end{equation*}
	for all $x \in \R^{\times}$. On the other hand, it follows from the assumptions on $\M uv_1$ and Proposition~\ref{prop:fourier-and-sobolev} that the function $x \mapsto \big(1 + \left\lvert\log\lvert x\rvert \right\rvert^{\alpha_2}\big) (\F_\times\M uv_1)(x)$ belongs to $L^2(\R^\times,\frac{\dxShort x}{\lvert x\rvert})$. Hence, we have
	\begin{equation*}
		\frac{\F_\times\M uv_1}{\lvert \mu_f\rvert^{\alpha_2/\alpha_1}} \in L^2(\R^\times,\frac{\dxShort x}{\lvert x\rvert}).
	\end{equation*}
	Let us denote the $L^2(\R^\times,\frac{\dxShort x}{\lvert x\rvert})$--norm of the latter function by $L$. Then it follows from Theorem~\ref{theo:error-bound-for-estimator} that
	\begin{align*}
		\err_0(n) \le & \frac{\err_1(n)}{a_n} + \frac{1}{2\sqrt{\pi}} \left\lVert \lvert \mu_f \rvert^{\frac{\alpha_2}{\alpha_1} - 1}\, \one_{\{ \lvert \mu_f\rvert \le a_n \} } \frac{1}{\lvert\mu_f\rvert^{\alpha_2/\alpha_1}} \F_\times\M uv_1 \right\rVert_{L^2(\R^\times,\frac{\dxShort x}{\lvert x\rvert})} \\
		\le & \frac{\err_1(n)}{a_n} +  \frac{L}{2\sqrt{\pi}} a_n^{\frac{\alpha_2}{\alpha_1} - 1}.
	\end{align*}
	This proves the corollary.
\end{proof}

\begin{rem}\label{rem:criterion-for-convergence}
	(a) Note that, in applications, one can try to compute the number $\alpha_1$ in the above corollary explicitly (in case that such a number exists) since the function $f$ -- and hence the functions $m_{f,\pm}$ -- are part of the model and thus known (cf. Example \ref{example:corr3_6} below).
	
	(b) In most cases, one will not know the exact mean quadratic error $\err_1(n)$ for the estimator $\widehat{uv_1}$ but only an upper bound $e_n$ for $\err_1(n)$. Of course, one has to choose $a_n$ such that $\frac{e_n}{a_n} \to 0$ to ensure convergence; cf.~Theorem~\ref{theo:error-bound-for-estimator}. If the number $\alpha_2$ were known, too (in case that such a number exists), one could optimize the choice of $a_n$ in order to obtain an optimal decay rate for the error bound for $\err_0(n)$ in Corollary~\ref{cor:criterion-for-convergence}.
	
	Indeed, it is not difficult to see that the optimal choice of $a_n$ is $a_n := \mathcal{O} (e_n^{\alpha_1/\alpha_2})$, which yields the decay rate $\mathcal{O} \Big(e_n^{1 - \frac{\alpha_1}{\alpha_2}}\Big)$ for the error bound of $\err_0(n)$. Unfortunately though, $\alpha_2$ will almost never be known. In fact, since $uv_1$ is unknown, we cannot even be sure whether a number $\alpha_2$ with the wanted properties exists. However, assuming its existence, one can simply choose $a_n := e_n^{1/2}$; this ensures that $\err_0(n)$ decays to $0$ (cf.~Theorem~\ref{theo:error-bound-for-estimator}) and it yields the decay estimate
	\begin{equation*}
		\err_0(n) \le e_n^{1/2} + C \, e_n^{(\frac{\alpha_2}{\alpha_1} - 1)/2}.
	\end{equation*}
\end{rem}

\begin{expl}\label{example:corr3_6}
	As mentioned in Remark~\ref{rem:criterion-for-convergence}, the function $f$ is assumed to be a part of the model 
	and thus, in some cases it is possible to compute $\alpha_1$ and $\gamma$ in Corollary~\ref{cor:criterion-for-convergence}.
	Let us give some examples of functions $f$ with $\alpha_1 = 1$:
	\begin{enumerate}[(a)]
		\item $\R \ni s \to f(s) = e^{-\theta |s|}$ for some $\theta > 0$; then 
		$$\gamma = \frac{2}{\theta} \left( 1 + \left( \beta + \frac{c-1}{2} \right)^2 \right)^{-1/2},$$
		provided that $\beta + \frac{c-1}{2} > 0$.
		\item $\R^2 \ni s \to f(s) = \tau \left( \kappa^2 - \left\| s \right\|_2^2  \right) \one_{(0,\kappa)}(\left\| s \right\|_2)$ for some
		$\tau,\kappa > 0$, where $\left\| \cdot \right\|_2$ denotes the Euclidean norm. Then
		$$\gamma = \pi \kappa^{2\beta+1+c} \tau^{\beta+(c-1)/2} \left( 1 + \left( 1 + \beta + \frac{c-1}{2} \right)^2  \right)^{-1/2},$$
		if it is assumed that $1 + \beta + \frac{c-1}{2} > 0$.
		\item $\R \ni s \to f(s) = \left( 1 + |s| \right)^{-\theta}$ for some $\theta > 0$; then
		$$\gamma =\frac{2}{\theta} \left( 1 + \left( \frac{1}{\theta} - \beta - \frac{c-1}{2} \right)^2 \right)^{-1/2},$$ 
		in case that $\beta + \frac{c-1}{2} > \frac{1}{\theta}$.
	\end{enumerate} 
	In all three cases, condition \eqref{eq:integrability-condition-stochastic-part} is evidently satisfied. The proofs are straightforward and we therefore omit them here. 
\end{expl}

\begin{expl}\label{ex:2.11}
	Let $D: \R_+ \to [0,1]$ be the distribution function of the gamma distribution $\Gamma\left(p, b+\frac{1}{2}\right)$ with  $p>0$,
	$b \ge 1$, i.e., let
	\begin{equation*}
		D(t) = \frac{(2b+1)^p}{2^p\Gamma (p)} \int_0^t s^{p-1} e^{-\big( b+\frac{1}{2} \big)s}ds, \quad t \geq 0.
	\end{equation*}
	Let $c=0$, $u(x)=x$. Then $f:[0,1) \to [1,\infty)$, given by the inverse of $t \mapsto D(\log t)$, satisfies condition \eqref{eq:integrability-condition-stochastic-part} with $\beta=c=1$ and condition~\eqref{eq:m_f_subbound_condition}
	with $\alpha_1 = p$. Indeed, $\int_0^1 f(t) dt=\E e^Y<\infty$, where a random variable $Y$ has distribution function $D$. Furthermore, a simple integral substitution yields that
	\begin{equation*}
		\begin{split}
			m_{f,+}(x) = & \frac{(2b+1)^p}{2^p\Gamma (p)} \int_0^\infty e^{-ixt} e^{\frac{1}{2}t} \frac{\dx}{\dx t}
			(f^{-1}(e^t)) \dx t.
		\end{split}
	\end{equation*}
	Taking into account that $f^{-1}(e^t) = D(t)$, the latter integral is a constant multiple of the characteristic function 
	of the Gamma distribution $\Gamma\left(p, b\right)$; hence $|m_{f+}(x)| \geq \gamma (1+|x|)^{-p}$ for all $x \in \R$ and 
	some $\gamma > 0$. 
\end{expl}

\subsection{Application to pure jump random fields}\label{subsec:pure_jump_random_fields}

We finally apply the above results to pure jump random fields. Let $u(x) = x$ and suppose
that the stationary infinitely divisible random field $X(t) = \int_{\R^d} f(t-x)\Lambda(dx), \ t\in \R^d$ has 
characteristic function $\varphi_{X(0)}$ given by
$$ \psi(y) := \varphi_{X(0)}(y) = \mathbb{E} e^{iyX(0)} = \exp \left\lbrace
\int\limits_{\R} \left( e^{iyx}  - 1 \right)v_1(x)dx  \right\rbrace, \quad y \in \R  . $$
Note that the numbers $a_0$, $b_0$ in the L\'{e}vy characteristics of $\Lambda$ are given by $a_0 = \int_{-1}^1 x v_0(x) dx$ and $b_0 = 0$ here.
Under the additional assumption $\int_{\R} |x| v_1(x) dx < \infty$ we have
$$ \psi^\prime (y) = i \psi(y) \int\limits_{\R} e^{iyx} x v_1(x)dx = i \psi(y) \F_+[uv_1](y), $$
which is equivalent to 
\begin{equation}\label{FourierRelation}
\F_+[uv_1](y) = -i \frac{\psi^\prime(y)}{\psi(y)}, 
\end{equation}
where $\F_+$ denotes the usual Fourier transform on the additive group $\R$.  
Now let $X$ be observed on 
a regular grid $\Delta \mathbb{Z}^d$ with mesh size $\Delta > 0$ (low frequency observations), i.e. consider the random field
$Y = \{Y_j ,\, j \in \Z^d\},$ where
\begin{equation}\label{eq:rfY}
	Y_j = X(\Delta j), \quad \Delta j = (\Delta j_1, \dots, \Delta j_d), \quad
	j = (j_1,\dots,j_d) \in \mathbb{Z}^d.
\end{equation}
For a finite nonempty set $W \subset \mathbb{Z}^d$
with cardinality $n = |W|$ let $(Y_{j})_{j \in W}$ be a sample drawn from $Y$. 
Based on relation~\eqref{FourierRelation}, an estimator for the Fourier
transform of $uv_1$ can be deduced by taking the empirical counterparts
\begin{equation*}
	\hat{\psi}(y)  = \frac{1}{n} \sum_{j \in W} e^{i y Y_j}, \quad \hat{\theta}(y) = \frac{1}{n} \sum_{j \in W}Y_j e^{i y Y_j}
\end{equation*}
of $\psi$ and $-i \psi^\prime$. In order to stabilize the estimator for small values of $\hat{\psi}$ we set
\begin{equation*}
	\widehat{\F_+[uv_1]}(y) =  \frac{\hat{\theta}(y)}{\tilde{\psi}(y)},
\end{equation*}
where
\begin{equation*}
	\frac{1}{\tilde{\psi}(y)} := \frac{1}{\hat{\psi}(y)}\id \{{|\hat{\psi}(y)| > n^{-1/2}}\}.
\end{equation*}
A natural idea now is to define an estimator $\widehat{uv_1}$ for $uv_1$ by taking the inverse Fourier transform
of $\widehat{\F_+[uv_1]}$. This may fail since in general $\widehat{\F_+[uv_1]}$ is not integrable. Nevertheless, it
is locally integrable, and thus we can define
\begin{equation}\label{eq:uv_1_estimator_pure_jump}
\widehat{uv_1}(x) = \F_+^{-1} \left[ \widehat{\F_+[uv_1]} \one_{[-\pi l, \pi l]} \right](x)
= \frac{1}{2\pi} \int\limits_{[-\pi l, \pi l]} e^{-iyx} \widehat{\F_+[uv_1]}(y)dy,
\end{equation}
for some $l > 0$. This estimator was originally designed by Comte and Genon-Catalot \cite{comte1} in case that $X$ is a L\'{e}vy process. Other direct estimation approaches for $v_1$ can be found in the  book \cite{LevMattersIV}.

The following theorem, which is due to \cite[Corollary 5.2]{ KarRothSpoWalk19}, gives an upper bound for the mean squared error $\err_1(n)$. 
\begin{theo}\label{eq:err1_bound}
	Let $Y=\{Y_j; \ j \in \Z^d \}$ be the random field introduced in~\eqref{eq:rfY}. Suppose $uv_1 \in H^{a}(\R,dx) \cap L^1(\R,dx)$ for some $a>0$. Moreover, let $\int_\R x^4 v_1(x)dx$ be finite and assume that there are numbers $b\geq 0$, $c_\psi > 0$ such that 
	$|\psi(x)| \geq c_\psi (1+x^2)^{-b/2}$ for all $x \in \R$. If either
	\begin{enumerate}
		\item the field $Y$ is $m$-dependent or
		\item the random field $Y$ is $\phi$-mixing such that
		\begin{equation}\label{eq:mixing_condition}
		\sum\limits_{j\in\Z^d\setminus \{0\}}\phi_{\infty,1}(|j|)<\infty \quad \text{and} \quad 
		\sum_{r=1}^\infty (r+1)^{3d-1}[\phi_{2,2}(r)]^{1/4}<\infty
		\end{equation}
	\end{enumerate}
	then
	\begin{equation*}
	\err_1^2(n) \leq \frac{K}{\big(1+(\pi l)^2\big)^a}   +K \frac{l}{n} \big(1+(\pi l)^2\big)^{b},
	\end{equation*}
	for all $l \in \N$ and some constant $K$.
\end{theo}
\begin{rem}
	\begin{enumerate}
		\item Since the mixing conditions~\eqref{eq:mixing_condition} are not further used in this paper, we omit the definitions of $\phi_{\infty,1}$ and $\phi_{2,2}$ here and refer the interested reader
		to \cite[Section 2.2]{ KarRothSpoWalk19} and references therein.
		\item We emphasize that the estimator $uv_0$ from Section~\ref{sec:estimate_v0} is not bound to a specific choice of the estimator for $uv_1$. Hence, the function $\widehat{uv_1}$ defined in~\eqref{eq:uv_1_estimator_pure_jump} can in general be any estimator that is at least contained in $L^2(\R^\times)$. Corollary~\ref{cor:criterion-for-convergence} yields sufficient conditions on $\widehat{uv_1}$ to be consistent in the mean square sense.
	\end{enumerate}
\end{rem}
It follows from Theorem~\ref{eq:err1_bound} that for an optimal choice of $l$, $\err_1(n) = \mathcal{O}(n^{-a/(2a+2b+1)})$ as $n \rightarrow \infty$. If additionally the conditions of 
Corollary~\ref{cor:criterion-for-convergence} are fulfilled, then the corresponding estimator $\widehat{uv_0}$
defined in~\eqref{eq:our-estimator} is consistent in the mean quadratic sense and (cf. Remark~\ref{rem:criterion-for-convergence})
\begin{equation*}
	\err_0(n) = \mathcal{O}\left( n^{  -\frac{a}{2a+2b+1}  \left( 1 - \frac{\alpha_1}{\alpha_2} \right) } \right),
	\quad \text{as} \ n \rightarrow \infty.
\end{equation*}

In order to find the numbers $b$ and $c_\psi$ above, the following theorem is quite helpful.

\begin{theo}
	There are numbers $b$, $c_\psi$, $C_\psi \geq 0$ such that $c_\psi (1+|x|)^{-b} \leq \psi(x) \leq C_\psi (1+|x|)^{-b}$ for all 
	$x \in \R$ if and only if
	\begin{equation}\label{eq:pol_decay_cond}
	-\log(C_\psi) + b \log (1+|x|) \leq \int_{0}^{x} \Imag \Big(\F_+ uv_1 \Big) (y) \dx y \leq -\log(c_\psi) + b \log (1+|x|)
	\end{equation}
	for all $x \in \R$, where $\Imag (z)$ denotes the imaginary part of a complex number $z$.
\end{theo}

\begin{proof}
	Since the function $uv_1 \in L^2(\R^\times, \dx x)$ is real-valued, we observe that
	\begin{equation*}
		\begin{split}
			|\psi(x)| = |\overline{\psi(x)}| & = \Big| \exp \Big( - \int_{\R} \frac{1 - e^{-ixy}}{y} (uv_1)(y) \dx y \Big) \Big| \\
			& = \Big| \exp \Big( - 2\pi i \int_\R (\F_+^{-1} \one_{[0,x]})(y) (uv_1)(y) \dx y  \Big) \Big|, \quad x \in \R. \\ 
		\end{split}
	\end{equation*}
	Applying the isometry property of the Fourier transform on $L^2(\R^\times, \dx x)$, we obtain
	\begin{equation*}
		\begin{split}
			|\psi(x)| & = \Big| \exp \Big( - i \int_{0}^x \overline{(\F_+uv_1)(y)} \dx y  \Big) \Big| \\
			& = \exp \Big( - \int_{0}^x \Imag \Big( (\F_+uv_1)(y) \Big) \dx y  \Big), \quad x \in \R.
		\end{split}
	\end{equation*}
	From this relation one immediately concludes the assertion of the theorem.
\end{proof}

\begin{rem}\label{rem:pol_decay_cond}
	\begin{enumerate}
		\item The bound in~\eqref{eq:pol_decay_cond} means that $\psi$ is polynomially decaying if and only if the function
		$\R \ni x \to \int_{0}^{x} \Imag \Big(\F_+ uv_1 \Big) (y) \dx y$ is either bounded or increases logarithmically.
		\item Notice that the lower bound in~\eqref{eq:pol_decay_cond} is equivalent to the upper bound for $\psi$ and vice versa.
		\item In fact, it may not always be possible to show existence of numbers $b$, $c_\psi$ and $C_\psi$ such that 
		inequality~\eqref{eq:pol_decay_cond} holds. Therefore, we refer to~\cite{Trabs14} (and the references therein) 
		for further conditions on $\psi$ to be polynomially decaying.
		\item Using relation~\eqref{eq:main-stochatic-part}, all the above mentioned regularity assumptions on $uv_1$ can immediately 
		be transferred to $uv_0$ and hence, implicitly are conditions on the random measure $\Lambda$. 
	\end{enumerate}
\end{rem}

\begin{expl}\label{ex:pure_jump_1}
	Let $d=1$, $c=0$, $f(s) = e^{-s} \one_{[0,\theta]}(s)$ for some $\theta > 0$. Moreover, consider the infinitely divisible random measure $\Lambda$
	with L\'{e}vy characteristics $a_0 = \pi^{-1/2} \int_{0}^{1} x^{1/2} e^{-x} \dx x$, $b_0 = 0$ and $v_0(x) = (\pi x)^{-1/2} e^{- x} \one_{(0,\infty)}(x)$,
	$x \in \R$. Then, the infinitely divisible moving average $ X(t)=\int_{t-\theta}^{t} e^{x-t} \Lambda(dx), \ t\in \R$
	is of pure jump structure with L\'{e}vy density 
	\begin{equation*}
		v_1(x) = \frac{1}{\sqrt{\pi}} x^{-1} \Gamma \Big( \frac{1}{2}, x, xe^{\theta} \Big) \one_{(0,\infty)}(x), \quad x \in \R,
	\end{equation*}
	where $\Gamma(x,y,z) = \int_{y}^z t^{x-1} e^{-t} \dx t$, $x > 0$ denotes the incomplete gamma function truncated at
	$y, \ z > 0$. Since the support of $f$ is bounded, $X$ is $m$-dependent with $m=\theta$. By Fubini's theorem, it follows that
	\begin{equation*}
		(\F_+ uv_1)(x) = \frac{1}{2} \int_{(0,\theta)} \frac{e^{s/2}}{(e^s - ix)^{3/2}} \dx s
	\end{equation*}
	and thus applying Minkowski's inequality for integrals shows that $uv_1 \in H^{a}(\R)$ for any $a < 1$. 
	In order to show existence of numbers $b \geq 0$ and $c_\psi > 0$ such that $|\psi(x)| \geq c_\psi (1+x^2)^{-b/2}$ for all $x \in \R$, it suffices to
	verify the upper bound in~\eqref{eq:pol_decay_cond} (cf. Remark~\ref{rem:pol_decay_cond}, (2)). Indeed, by Fubini's theorem, we observe that
	\begin{equation*}
		\begin{split}
			\int_0^x \F_+ uv_1 (y) \dx y & = \frac{1}{2} i \Big( \theta - \int_0^\theta \frac{e^{s/2}}{(e^s - ix)^{1/2}} \dx s \Big) \\
			& = i \Big( \theta - 2\log(e^{\theta/2} + \sqrt{e^\theta -ix}) + 2\log(1 + \sqrt{1 -ix}) \Big),
			\quad x \in \R.
		\end{split}
	\end{equation*}
	Hence, 
	\begin{equation*}
		\int_0^x \Imag \Big( \F_+ uv_1 \Big) (y) \dx y = \theta + 2 \log \Big| \frac{1 + \sqrt{1 -ix}}{e^{\theta/2} + \sqrt{e^\theta -ix}} \Big|
		\to \theta \quad \text{as} \ x \to \pm \infty,
	\end{equation*}
	i.e. the integral is bounded; consequently $b=0$ is an appropriate choice. \\
	Note finally that 
	\begin{equation*}
		(\M uv_1)(-e^x) = 0 \quad \text{and} \quad (\M uv_1)(e^x) = \frac{1}{\sqrt{\pi}} e^{x/2} \Gamma \Big( \frac{1}{2}, e^{x}, e^{x + \theta} \Big)
	\end{equation*}
	for all $x \in \R$; thus both of $(\M uv_1)(\exp(\argument))$ and $(\M uv_1)(-\exp(\argument))$ are contained in the Sobolev space of any order.
	Hence, Corollary \ref{cor:criterion-for-convergence} applies in this setting and the estimators $\widehat{uv_0}$, $\widetilde{uv_0}$ for $uv_0$ 
	defined in~\eqref{eq:our-estimator},~\eqref{eq:alternative_estimator} respectively, are consistent in the mean quadratic sense with the rate of convergence given by 
	$$n^{  -\frac{a}{2a+1}  \left( 1 - \frac{1}{\alpha_2} \right)},$$
	as $n \rightarrow \infty$, for fixed $0 < a < 1$, $\alpha_2 > 1$.	
\end{expl}

\begin{expl}\label{ex:pure_jump_2}
	Suppose $d=2$, $c=0$ and let $f(s)$ be as in Example~\ref{example:corr3_6}, (b). Moreover, consider the infinitely divisible random measure $\Lambda$
	with the same L\'{e}vy characteristics $a_0$, $b_0$ and $v_0$ as in Example~\ref{ex:pure_jump_1}. Then, the infinitely 
	divisible moving average random field $ X(t)=\int_{\|x-t\|_2 \leq \kappa} \tau (\kappa^2 - \left\| x-t \right\|_2^2) \Lambda(dx), \ t\in \R^2$
	is of pure jump structure with L\'{e}vy density 
	\begin{equation*}
		v_1(x) = \frac{\sqrt{\pi}}{\tau} \int_{\frac{x}{\tau \kappa^2}}^{\infty} r^{-3/2} e^{-r} \dx r \one_{(0,\infty)}(x), \quad x \in \R.
	\end{equation*}
	As in the previous example $X$ is $m$-dependent since the support of $f$ is bounded, where $m = 2\kappa$. By Fubini's theorem, it follows that
	\begin{equation*}
		(\F_+ uv_1)(x) = \frac{\sqrt{\pi}}{\tau} \int_{(0,\infty)} r^{-3/2} e^{-r} \frac{e^{ixr\tau \kappa^2} - 1 - ixr\tau \kappa^2 e^{ixr\tau \kappa^2}}{x^2} \dx r.
	\end{equation*}
	Thus, $\F_+ uv_1$ is bounded at $x=0$ and of order $\mathcal{O}(x^{-1})$ at infinity, i.e. $uv_1 \in H^a(\R,dx)$ for any $a < 1/2$. \\
	In order to show existence of numbers $b \geq 0$ and $c_\psi > 0$ such that $|\psi(x)| \geq c_\psi (1+x^2)^{-b/2}$ for all $x \in \R$, it suffices to
	verify the upper bound in~\eqref{eq:pol_decay_cond} (cf. Remark~\ref{rem:pol_decay_cond}, (2)). Indeed, by Fubini's theorem, we observe that
	\begin{equation*}
		\begin{split}
			\int_0^x \Imag \Big(\F_+ uv_1 \Big) (y) \dx y & = \frac{\sqrt{\pi}}{\tau} \int_0^\infty r^{-3/2} e^{-r}
			\int_0^x \frac{\sin (y r \tau \kappa^2) - y r \tau \kappa^2 \cos (y r \tau \kappa^2)}{y^2} \dx y \dx r \\
			& = \sqrt{\pi} \kappa^2 \int_0^\infty r^{-1/2} e^{-r} \Big( 1 - \frac{\sin (x r \tau \kappa^2)}{x r \tau \kappa^2} \Big) \dx r,
			\quad x \in \R.
		\end{split}
	\end{equation*}
	Hence, $\Big| \int_0^x \Imag \Big(\F_+ uv_1 \Big) (y) \dx y \Big| \leq c$ for some constant $c > 0$ and all $x \in \R$; consequently
	$b=0$ is an appropriate choice. \\
	Finally, we observe that 
	\begin{equation*}
		(\M uv_1)(-e^x) = 0 \quad \text{and} \quad (\M uv_1)(e^x) = e^{x/2} \frac{\sqrt{\pi}}{\tau} \int_{\frac{e^x}{\tau \kappa^2}}^{\infty} r^{-3/2} e^{-r} \dx r 
	\end{equation*}
	for all $x \in \R$; thus, both of $(\M uv_1)(\exp(\argument))$ and $(\M uv_1)(-\exp(\argument))$ belong to the Sobolev space of any order.
	Hence, Corollary \ref{cor:criterion-for-convergence} applies in this setting, and the estimators $\widehat{uv_0}$, $\widetilde{uv_0}$ for $uv_0$ 
	defined in~\eqref{eq:our-estimator},~\eqref{eq:alternative_estimator} respectively, are consistent in mean quadratic sense with the rate of convergence given by 
	$$n^{  -\frac{a}{2a+1}  \left( 1 - \frac{1}{\alpha_2} \right)},$$
	as $n \rightarrow \infty$, for fixed $0 < a < 1/2$, $\alpha_2 > 1$.
\end{expl}

\begin{rem}\label{rem:comp_time}
	The proposed methods are applicable for arbitrary dimensions $d \in \N$. Nevertheless, for computational reasons,
	we restricted the above examples to $d=1,2$. In fact, simulation of the random field $X$ on a large lattice,
	together with the numerical computation of various integral transforms in the formula for the solution of equation~\eqref{eq:main-stochatic-part} 
	may be very time consuming, cf. Tables \ref{table:result_proc} and \ref{table:result_field}.
\end{rem}


\section{{A general} integral equation} \label{sec:integral-equation}

Here we discuss existence and uniqueness of solution for integral equation~\eqref{eq:integ_equation}. Throughout the entire Section~\ref{sec:integral-equation}, fix a number $c \in \R$ and measurable functions $g,h:\R^d \rightarrow \R$. 

\subsection{The setting}\label{sec:setting_int_eq}

We study the integral equation 
\begin{equation}\label{eq:main}
v(x) = \int_{\supp(h)} g(s) \; w\big(h(s) x\big) \dx s, \qquad x \in \R^\times
\end{equation} 
on the space $L^2(\R^\times, \lvert x\rvert^c \dx x)$; more precisely, for any given $v \in L^2(\R^\times, \lvert x\rvert^c \dx x)$ we give necessary and sufficient conditions for~\eqref{eq:main} to have a unique solution $w$ in the same space.

We start by showing that, under appropriate assumptions on $g$, $h$ and $c$, the right hand side of~\eqref{eq:main} can be interpreted as an integral operator on $L^2(\R^\times, \lvert x\rvert^c \dx x)$.

\begin{prop}\label{prop:operator-on-L_2}
	Assume that the functions $g$, $h$ and the number $c$ satisfy the integrability condition
	\begin{equation}
	\label{eq:integrability-condition}
	C := \int\limits_{\supp(h)} \vert g(s)\rvert \, \lvert h(s)\rvert^{-(c+1)/2} \dx s < \infty.
	\end{equation}
	Then, for each $w \in L^2(\R^\times, \lvert x\rvert^c \dx x)$, the function $s \to g(s) \, w\big(h(s)x\big)$ is integrable over $\supp(h)$ for almost every $x \in \R$. Moreover, the mapping 
	\begin{equation*}
	\Gi: w \mapsto \int\limits_{\supp(h)} g(s) \; w\big(h(s) \argument \big) \dx s 
	\end{equation*}
	is a bounded linear operator on the space $L^2(\R^\times, \lvert x\rvert^c \dx x)$ and has operator norm $\lVert \Gi\rVert \le C$.
\end{prop}
\begin{proof}
	For each $w \in L^2(\R^\times, |x|^c dx)$ Minkowski's integral inequality yields
	\begin{align*}
	\lVert \Gi w \rVert_{L^2(\R, \lvert x\rvert^c \dx x)}
	& = \left( \int\limits_{\R}\left( \int\limits_{\supp(h)} \left\lvert g(s) \, w\big(h(s)x\big) \right\rvert \dx s \right)^2 \lvert x\rvert^c \dx x \right)^{1/2} \\
	& \leq \int\limits_{\supp(h)} |g(s)| \left(  \int\limits_{\R} \left\lvert w\big(h(s)x\big) \right\rvert^2 \lvert x\rvert^c \dx x \right)^{1/2} \dx s \\
	& = \lVert w\rVert_{L^2(\R, \lvert x\rvert^c \dx x)} \int\limits_{\supp(h)} \lvert g(s)\rvert \lvert h(s)\rvert^{-(c+1)/2} \dx s < \infty
	\end{align*}
	where the equality between the last two lines follows by a simple integral substitution. This proves the assertions.
\end{proof}

Assume that the integrability condition~\eqref{eq:integrability-condition} holds. Then we can rephrase our integral equation~\eqref{eq:main} in terms of the operator $\Gi$ on $L^2(\R, \lvert x\rvert^c \dx x)$: for given $v \in L^2(\R, \lvert x\rvert^c \dx x)$ our equation has a solution $w$ in the same space if and only if $v$ is contained in the range of $\Gi$. On the other hand, an existing solution is unique if and only if the operator $\Gi$ is injective.

In order to analyse the range and the kernel of $\Gi$ we use techniques from classical harmonic analysis. The structure of $\Gi$ allows us to transform this operator into a multiplication operator by using the Fourier transform on the multiplicative group $\R^\times = \R  \setminus \{0\}$. To this end, we recall a few facts about harmonic analysis on $\R^\times$ in the next subsection.

\subsection{The Fourier transform on $\R^\times$} \label{subsec:fourier-transform-on-multiplicative-real-numbers}

Since the properties of the Fourier transform on the multiplicative group of non-zero real numbers seem to be hardly treated in detail in the literature, we include here a brief overview about them. As explained below, those properties are not difficult to derive from well-known properties of the Fourier transform on the additive groups $\R$ and $\Z/2\Z$, but for the convenience of the reader we write them down explicitly.

When endowed with the usual multiplication and the Euclidean topology, the set $\R^\times = \R \setminus \{0\}$ is a locally compact abelian group and the measure $\frac{\dx x}{|x|}$ is the (unique up to scalar multiples) Haar measure on this group. The \emph{dual group} $\widehat{\R^\times}$ consists by definition of all continuous group homomorphisms from $\R^\times$ into the complex unit circle $\T$; as usual, we endow the dual group with the topology of pointwise convergence. Then the dual group $\widehat{\R^\times}$ is also a locally compact abelian group and it is isomorphic to $\R^\times$ itself in the following way: For each $x \in \R^\times$ we define
\begin{equation*}
\delta(x) := 
\begin{cases}
0 \quad & \text{if } x > 0, \\
1 \quad & \text{if } x < 0.
\end{cases}
\end{equation*}
Then for every $y \in \R^\times$ the mapping
\begin{align*}
\xi_y: \R^\times \to \T, \quad \xi_y(x) = e^{i \log\lvert x\rvert \cdot \log \lvert y\rvert} \cdot e^{i\pi \delta(x) \delta(y)} 
\end{align*}
is an element of the dual group $\widehat{\R^\times}$. Moreover, all elements of $\widehat{\R^\times}$ arise in this way, and $\R^\times \ni y \to \xi_y \in \widehat{\R^\times}$ is an isomorphism of the locally compact abelian groups $\R^\times$ and $\widehat{\R^\times}$. 

In fact, all assertions above can easily be concluded from well-known properties of the additive groups $\R$ and $\Z/2\Z = \{\overline{0},\overline{1}\}$: It is a standard fact in harmonic analysis that the dual groups of $\R$ and $\Z/2\Z$ are isomorphic to those groups themselves, respectively. More precisely, the dual group of $\R$ consists of all mappings of the form
\begin{align*}
\R \ni x \mapsto e^{ixy} \in T
\end{align*}
for fixed $y \in \R$, while the dual group of $\Z/2\Z$ consists of all mappings of the form
\begin{align*}
\Z/2\Z \ni \overline{x} \mapsto e^{i\pi xy} \in \T
\end{align*}
for fixed $\overline{y} \in \Z/2\Z$. Since the multiplicative group $\R^\times$ is isomorphic to the product group $\R \times \Z/2\Z$ via the group isomorphism
\begin{equation*}
\R^\times \ni x \to (\log\lvert x\rvert, \overline{\delta(x)}) \in \R \times \Z / 2\Z,
\end{equation*}
one easily concludes that the dual group of $\R^\times$ has the properties claimed above (use that $\widehat{G \times H}$ is isomorphic to $\hat{G} \times \hat {H}$ for all locally compact abelian groups!).

From now on, we identify the multiplicative group $\R^\times$ and its dual group $\widehat{\R^\times}$ via the isomorphism $y \to \xi_y$. The Fourier transform on the group $\R^\times$ thus becomes a bijective linear mapping
\begin{equation*}
\F_\times: L^2(\R^\times, \frac{\dxShort x}{\lvert x\rvert}) \to L^2(\R^\times, \frac{\dxShort x}{\lvert x\rvert}).
\end{equation*}
Using the well-known formulas for the Fourier transform on the groups $\R$ and $\Z/2\Z$ and, again, the fact that $\R^\times$ is isomorphic to $\R \times \Z/2\Z$, we also obtain a formula for the Fourier transform $\F_\times$ on the multiplicative group $\R^\times$: for all functions $u \in L^2(\R^\times, \frac{\dxShort x}{\lvert x\rvert}) \cap L^1(\R^\times, \frac{\dxShort x}{\lvert x\rvert})$ we can compute $\F_\times u$ explicitly by means of the formula
\begin{equation} \label{eq:fourier-transform}
(\F_\times u)(y) = \int_{\R^\times} u(x) \; e^{-i\log\lvert x\rvert \cdot \log \lvert y \rvert} \cdot e^{i\pi \delta(x) \delta(y)} \, \frac{\dxShort x}{\lvert x\rvert}.
\end{equation}
We point out that the multiple $\frac{1}{2\sqrt{\pi}} \F_\times$ of $\F_\times$ is a unitary operator on the Hilbert space $L^2(\R^\times, \frac{\dxShort x}{\lvert x\rvert})$ and that the Fourier inversion formula
\begin{equation} \label{eq:inverse-multiplicative-fourier-transform}
	\frac{1}{4\pi} (\F_\times^2 u)(x) = u(1/x)
\end{equation}
holds for all $u \in L^2(\R^\times, \frac{\dxShort x}{\lvert x\rvert})$. 

The smoothness of a function is closely related to the growth behaviour of its Fourier transform: indeed, for each $\alpha \in [0,\infty)$, a function $f \in L^2(\R,\dxShort x)$ is contained in the Sobolev space $H^\alpha(\R, \dx x)$ if and only if the Fourier transform of $f$ (with respect to the additive group $\R$), multiplied with $1+|x|^\alpha$, is contained in $L^2(\R,\dxShort x)$. From this, one immediately obtains an analogue result on the multiplicative group $\R^\times$ which we now state explicitly for the sake of later reference.

\begin{prop} \label{prop:fourier-and-sobolev}
	Fix $\alpha \in [0,\infty)$. For each $u \in L^2(\R^\times, \frac{\dxShort x}{x})$ the following assertions are equivalent:
	\begin{enumerate}
		\item[(i)] Both the functions $u\big(\exp(\argument)\big)$ and $u\big(-\exp(\argument)\big)$ belong to the Sobolev space $H^\alpha(\R, \dx x)$.
		\item[(ii)] The function $x \mapsto \big(1 + \left\lvert\log\lvert x\rvert \right\rvert^\alpha\big) (\F_\times u)(x)$ is contained in $L^2(\R^\times, \frac{\dxShort x}{|x|})$.
	\end{enumerate}
\end{prop}

\begin{rem}
	We note that the Fourier transform on the multiplicative groups $(0,\infty)$ and $\R^\times$ is closely related to another important integral transform, namely the \emph{Mellin transform}. Our results in the next subsection can also be stated (and proved) in terms of the Mellin transform. Recent explicit examples for the usage of the Mellin transform in the statistics of L\'evy processes can be found in \cite{BelomSchoenm16,BelGol18}.
	In the latter paper, an integral equation similar to~\eqref{eq:main} is treated, and in \cite{BelomSchoenm16}  it is used for density estimation of the time change in a L\'evy process.
	
	The reason why we prefer to use the Fourier transform on $\R^\times$ in this paper is mainly a notational one: the Mellin transform is defined for functions $(0,\infty) \to \C$, so if we wish to analyse a function $\varphi: \R^\times \to \C$, we would have to consider two Mellins transforms -- one of the function $(0,\infty) \ni x \mapsto \varphi(x) \in \C$ and one of the function $(0,\infty) \ni x \mapsto \varphi(-x) \in \C$. This would make many of the Hilbert space isomorphisms that we employ in the subsequent section even more complicated to write down.
\end{rem}

We will use the Fourier transform $\F_\times$ to show that the operator $\Gi$ defined in Proposition~\ref{prop:operator-on-L_2} is unitarily similar to a multiplication operator.


\subsection{Solution theory for our integral equation} \label{subsection:solution-theory}

In what follows, we give necessary and sufficient criteria for the intregral equation~\eqref{eq:main} to have a unique solution for all $v$. Recall from Proposition~\ref{prop:operator-on-L_2} that $\Gi$ is a bounded linear operator on $L^2(\R, \lvert x\rvert^c\dx x) = L^2(\R^\times, \lvert x\rvert^c\dx x)$
and that the Fourier transform $\F_\times$ is an operator on $L^2(\R^\times,\frac{\dxShort x}{\lvert x \rvert})$. Thus, in order to combine the Fourier transform $\F_\times$ with the operator $\Gi$ we have to intertwine these two maps by a similarity transform between those two Hilbert spaces. To this end, note that the mapping
\begin{equation*}
\M: L^2(\R^\times, \lvert x\rvert^c \dx x) \to L^2(\R^\times, \frac{\dxShort x}{\lvert x\rvert}), \qquad (\M u)(x) = \lvert x\rvert^{(c+1)/2} \, u(x)
\end{equation*}
is a unitary linear operator; this can be checked by a brief computation. Now, we define a linear operator
\begin{equation*}
\tilde \Gi := \F_\times \M \Gi \M^{-1} \F_\times^{-1} = (\frac{1}{2\sqrt{\pi}}\F_\times)\,  \M \Gi \M^{-1} \, (\frac{1}{2\sqrt{\pi}}\F_\times)^{-1}
\end{equation*}
on the space $L^2(\R^\times, \frac{\dxShort x}{\lvert x\rvert})$. The definition of $\tilde \Gi$ is well illustrated by the following commutative diagram:
\begin{equation} \label{eq:commutative-diagram}
\begin{tikzcd}
L^2(\R^\times, \frac{\dxShort x}{\lvert x\rvert}) \arrow{d}{(\frac{1}{2\sqrt{\pi}}\F_\times)^{-1}} \arrow{rrr}{\tilde \Gi} & & & L^2(\R^\times, \frac{\dxShort x}{\lvert x\rvert})\\
L^2(\R^\times, \frac{\dxShort x}{\lvert x\rvert}) \arrow{d}{\M^{-1}} & & & L^2(\R^\times, \frac{\dxShort x}{\lvert x\rvert}) \arrow{u}{\frac{1}{2\sqrt{\pi}} \F_\times} \\
L^2(\R^\times, \lvert x\rvert^c \dx x) \arrow{rrr}{\Gi} & & & L^2(\R^\times, \lvert x\rvert^c \dx x) \arrow{u}{\M}
\end{tikzcd}
\end{equation}
As $\M$ and $\frac{1}{2\sqrt{\pi}} \F_\times$ are unitaries, the operators $\Gi$ and $\tilde \Gi$ are similar in the sense that they are intertwined by a Hilbert space isomorphism. In particular, $\Gi$ is injective if and only if $\tilde \Gi$ is injective. Moreover, our integral equation~\eqref{eq:main}, which can be rewritten as $\Gi w = v$, is equivalent to $\tilde \Gi \F_\times \M w  = \F_\times \M v$. Hence, the integral equation has a solution if and only if $\F_\times \M v$ is contained in the range of $\tilde \Gi$.

Now, the point is that the operator $\tilde \Gi$ has a very simple structure: in fact $\tilde \Gi$ is a multiplication operator, as explained in the subsequent theorem. To state the theorem, the following three functions are important: whenever the integrability condition~\eqref{eq:integrability-condition} is satisfied, we define $m_+,m_-: \R \to \C$ and $\mu: \R^\times \to \C$ by
\begin{equation}
\label{eq:functions-m}
\begin{split}
m_+(x) & := \int\limits_{\supp(h)} g(s) \, \lvert h(s)\rvert^{-(c+1)/2} \, e^{i x \, \log \lvert h(s)\rvert} \dx s, \\
m_-(x) & := \int\limits_{\supp(h)} g(s) \, \lvert h(s)\rvert^{-(c+1)/2} \, e^{i x \, \log \lvert h(s)\rvert} \, \sgn h(s) \dx s, \\
\mu(y) & :=
\begin{cases}
m_+(\log\lvert y\rvert) \quad & \text{if } y > 0, \\
m_-(\log\lvert y\rvert) \quad & \text{if } y < 0.
\end{cases}
\end{split}
\end{equation}
The functions $m_+$ and $m_-$ are bounded and continuous functions from $\R$ to $\C$ (the continuity follows from the dominated convergence theorem) and hence, $\mu$ is bounded and continuous from $\R^\times$ to $\C$.

\begin{theo} \label{theo:similar-to-multiplication-operator}
	Assume that the integrability condition~\eqref{eq:integrability-condition} is satisfied. Then the linear operator $\tilde \Gi$ on $L^2(\R^\times, \frac{\dxShort x}{\lvert x\rvert})$ is given by
	\begin{equation*}
	\tilde \Gi u = \mu u \qquad \text{for all } u \in L^2(\R^\times, \frac{\dxShort x}{\lvert x\rvert});
	\end{equation*}
	here, $\mu: \R^\times \to \C$ is the bounded and continuous function defined in~\eqref{eq:functions-m}.
\end{theo}
\begin{proof}
	We have to show that $\tilde G u = \mu u$ for all $u \in L^2(\R^\times, \frac{\dxShort x}{\lvert x\rvert})$. Note that this is true if and only if $\tilde \Gi \F_\times \M w = \mu \F_\times\M w$, $w \in L^2(\R^\times, \lvert x\rvert^c \dx x)$, and this is in turn equivalent to $\F_\times \M \Gi w = \mu \F_\times\M w$ for all $w \in L^2(\R^\times, \lvert x\rvert^c \dx x)$. In order to prove this equality we may, by a simple density argument, assume that $w$ is a continuous mapping from the topological space $\R^\times$ to $\C$ and that the support of $w$ is contained in a compact subset of $\R^\times$ (which, in particular, implies that $w$ vanishes close to $0$).
	Then $\M w$ is also a continuous function whose support is contained in a compact subset of $\R^\times$, so $\M w$ is contained in $L^1(\R^\times, \frac{\dxShort x}{\lvert x\rvert})$ and hence we may use formula~\eqref{eq:fourier-transform} to compute its Fourier transform $\F_\times\M w$. We obtain
	\begin{equation*}
	(\F_\times\M w)(y) = \int_{\R^\times} \lvert x\rvert^{(c-1)/2} w(x) \, e^{-i\log\lvert x\rvert \, \log\lvert y\rvert} \cdot e^{i\pi \delta(x) \delta(y)} \, \dx x.
	\end{equation*}	
	Let us now show that $\M\Gi w$ is also contained in $L^1(\R^\times, \frac{\dxShort x}{\lvert x\rvert})$, so that $\F_\times \M \Gi w$ can be computed by formula~\eqref{eq:fourier-transform}, too. Indeed, one easily checks that
	\begin{equation*}
	\int\limits_{\R^\times} \left\lvert (\M \Gi w)(x)\right\rvert \, \frac{\dxShort x}{\lvert x\rvert} \le \int\limits_{\supp(h)} \lvert g(s)\rvert \lvert h(s)\rvert^{-(c+1)/2} \dx s \; \int\limits_{\R^\times} \lvert w(x)\rvert \lvert x\rvert^{(c-1)/2} \dx x.
	\end{equation*}
	The latter expression is finite since the integrability condition~\eqref{eq:integrability-condition} is satisfied and since $w$ is continuous and its support is contained in a compact subset of $\R^\times$. Hence, $\M \Gi w \in L^1(\R^\times, \frac{\dxShort x}{\lvert x\rvert})$ and formula~\eqref{eq:fourier-transform} is thus applicable in order to compute $\F_\times \M \Gi w$; the formula yields
	\begin{equation*}
	(\F_\times \M \Gi w)(y) = \int\limits_{\R^\times} \lvert x \rvert^{(c-1)/2} \, \int\limits_{\supp(h)} g(s) \; w\big(h(s) x \big) \dx s \, e^{-i\log\lvert x\rvert \, \log \lvert y\rvert} \cdot e^{i\pi \delta(x) \delta(y)} \, \dx x.
	\end{equation*}
	A similar estimate as above shows that, for each $y \in \R^\times$, the entire integrand of the preceding integral is contained in $L^1(\R^\times \times \supp(h), \frac{\dxShort x}{\lvert x \rvert} \times \dx s)$; we may thus use Fubini's theorem and obtain by a simple substitution that $(\F_\times \M \Gi w)(y)$ is given by
	\begin{align*}
	\int\limits_{\supp(h)} g(s) \, \lvert h(s)\rvert^{-(c+1)/2} \, \int\limits_{\R^\times} \lvert x\rvert^{(c-1)/2} w(x) \, e^{-i\log\frac{\lvert x\rvert}{\lvert h(s) \rvert} \, \log\lvert y\rvert} \cdot e^{i\pi \delta(\frac{x}{h(s)}) \delta(y)} \, \dx x \dx s
	\end{align*}
	for almost all $y \in \R^\times$. Note that $\delta(\frac{x}{h(s)})$ equals $\delta(x) - \delta(h(s))$ modulo $2$, and the latter number equals $\delta(x) + \delta(h(s))$ modulo $2$. This proves that
	\begin{equation*}
	e^{i\pi \delta(\frac{x}{h(s)}) \delta(y)} = e^{i\pi \delta(x) \delta(y)} \cdot e^{i \pi \delta(h(s))\delta(y)}.
	\end{equation*}
	Hence, $(\F_\times \M \Gi w)(y)$ is the product of $(\F_\times \M w)(y)$ with
	\begin{align*}
	\int\limits_{\supp(h)} g(s) \, \lvert h(s)\rvert^{-(c+1)/2} \, e^{i \log \lvert y\rvert \, \log \lvert h(s)\rvert} \cdot e^{i \pi \delta(h(s))\delta(y)} \dx s,
	\end{align*}
	and this function is easily checked to equal $\mu$.
\end{proof}

Since the operators $\Gi$ and $\tilde \Gi$ are unitarily similar, Theorem~\ref{theo:similar-to-multiplication-operator} immediately yields the following conditions for existence and uniqueness of solutions to our integral equation~\eqref{eq:main}.

\begin{cor} \label{cor:injectivit-and-surjectivity-of-G}
	Assume that the integrability condition~\eqref{eq:integrability-condition} is satisfied and let $\mu: \R^\times \to \C$ be the bounded and continuous function defined in~\eqref{eq:functions-m}. The operator $\Gi$ on $L^2(\R^\times, \lvert x\rvert^c \dx x)$ from Proposition~\ref{prop:operator-on-L_2} is  
	\begin{enumerate}
		\item[(a)]  injective if and only if $m_+ \not= 0$ and $m_- \not= 0$ almost everywhere on $\R$ (with respect to the 
		Lebesgue measure).
		\item[(b)]   surjective if and only if it is bijective if and only if $\inf_{x \in \R} \lvert m_+(x)\rvert > 0$ and $\inf_{x \in \R} \lvert m_-(x)\rvert > 0$.
	\end{enumerate}
\end{cor}

The condition $\inf_{y \in \R^\times} \lvert \mu(y)\rvert > 0$ in part~(b) of the above corollary is rather restrictive, and in applications it happens quite frequently that $\Gi$ is not surjective. However, in order to solve our integral equation~\eqref{eq:main} we do not really need $\Gi$ to be surjective -- it suffices, of course, if $v$ is contained in the range of $\Gi$. Thus, the following corollary is quite useful.

\begin{cor} \label{cor:criterion-for-solvability}
	Fix $\alpha \in [0,\infty)$. Assume that the integrability condition~\eqref{eq:integrability-condition} is fulfilled and suppose that the two functions $m_+,m_-: \R \to \C$ defined in~\eqref{eq:functions-m} satisfy the estimate
	\begin{equation}\label{eq:lower_bound_m}
	\lvert m_{\pm}(x)\rvert \ge \frac{\gamma}{1 + \lvert x\rvert^\alpha}
	\end{equation}
	for all $x \in \R$ and a constant $\gamma > 0$. If $v \in L^2(\R^\times, \lvert x\rvert^c\dx x)$ and 
	if both functions
	\begin{equation*}
	(\M v)(\exp(\argument)) \qquad \text{and} \qquad (\M v)(-\exp(\argument))
	\end{equation*}
	are contained in the Sobolev space $H^\alpha(\R, \dx x)$, then our integral equation~\eqref{eq:main} has a unique solution $w \in L^2(\R^\times, \lvert x\rvert^c\dx x)$ given by $w =  \M^{-1} \F_\times^{-1} \big(\frac{1}{\mu} \F_\times \M v\big)$.
\end{cor}


\begin{proof}[Proof of Corollary~\ref{cor:criterion-for-solvability}]
	It follows from Corollary~\ref{cor:injectivit-and-surjectivity-of-G}(a) and the estimate on $m_{\pm}$ in the assumptions that the solution of our integral equation is unique whenever it exists. To prove the existence of a solution, let $v \in L^2(\R^\times, \lvert x\rvert^c\dx x)$ satisfy the conditions stated in the assertion of the corollary. We only have to show that $\F_\times \M v$ is contained in the range of the multiplication operator $\tilde G$ on $L^2(\R^\times, \frac{\dxShort x}{\lvert x\rvert})$; in this case, the solution $w$ clearly exists and is of the claimed form. Since $(\M v)(\exp(\argument))$, $(\M v)(-\exp(\argument)) \in H^\alpha(\R, \dx x)$, Proposition~\ref{prop:fourier-and-sobolev} yields that the function
	\begin{equation*}
	x \mapsto \big(1 + \left\lvert\log\lvert x\rvert \right\rvert^\alpha\big) (\F_\times\M v)(x)
	\end{equation*}
	is contained in $L^2(\R^\times, \frac{\dxShort x}{\lvert x\rvert})$. On the other hand, the assumed estimate for $m_+$ and $m_-$ implies that $\frac{1}{\lvert\mu(x)\rvert} \le \frac{1 + \left\lvert\log\lvert x\rvert \right\rvert^\alpha}{\gamma}$ for all $x \in \R^\times$. Thus, $\frac{\F_\times \M v}{\mu} \in L^2(\R^\times, \frac{\dxShort x}{\lvert x\rvert})$ and hence, $\F_\times\M v = \tilde \Gi \frac{\F_\times \M v}{\mu}$ is contained in the range of $\tilde \Gi$. This proves the assertion.
\end{proof}

Note that, if we set $\alpha = 0$ in Corollary~\ref{cor:criterion-for-solvability}, then we see that the condition $\inf_{y \in \R^\times} \lvert \mu(y)\rvert > 0$ implies that the operator $\Gi$ is surjective. Hence, Corollary~\ref{cor:criterion-for-solvability} can be seen as a refinement of the solvability criterion in Corollary~\ref{cor:injectivit-and-surjectivity-of-G}(b). Let us also point out the following observation concerning the injectivity of the operator $\Gi$:

\begin{rem} \label{rem:holomorphic-case}
	Assume that there exist constants $\gamma_2 \ge \gamma_1 > 0$ such that
	\begin{equation*}
	\gamma_1 \le h(s) \le \gamma_2 \qquad \text{for almost all } s \in \supp(h).
	\end{equation*}
	Then the integrability condition~\eqref{eq:integrability-condition} is equivalent to $g|_{\supp(h)} \in L^1(\supp(h))$, so suppose for the rest of this remark that $g|_{\supp(h)} \in L^1(\supp(h))$. 
	
	The assumptions that we just imposed on $h$ imply that $\sup_{s \in \supp(h)} \left\lvert \log \lvert h(s)\rvert \right\rvert < \infty$. From this, it easily follows that the integrals in~\eqref{eq:functions-m} that define $m_+(x)$ and $m_-(x)$ also make sense for $x \in \C$. We thus obtain functions $\C \ni x \mapsto m_{\pm}(x) \in \C$ and it is easy to see that those functions are analytic. Hence, if neither $m_+$ nor $m_-$ is identically $0$ on $\C$, then both of those functions vanish at at most countably many points. It thus follows from Corollary~\ref{cor:injectivit-and-surjectivity-of-G}(a) that the operator $\Gi$ from Proposition~\ref{prop:operator-on-L_2} is injective.
	
	To give a more concrete example	of such situation,
	assume that there exists a number $t_1 \in \R$ for which $0 \not= \int_{\supp(h)} g(s) \lvert h(s) \rvert^{t_1} \dx s$; the latter integral equals $m_+\big(-i(t_1 + \frac{c+1}{2})\big)$, so it follows that $m_+$ is not identically $0$ on $\C$. Also assume that there exists a number $t_2 \in \R$ for which we have $0 \not= \int_{\supp(h)} g(s) \lvert h(s) \rvert^{t_2} \, \sgn h(s) \dx s$. The latter integral equals $m_-\big(-i(t_2 + \frac{c+1}{2})\big)$ and thus $m_-$ is not identically $0$ on $\C$. Hence, we can conclude that $\Gi$ is injective. This argument shows in particular that $\Gi$ is injective if we have both $\int_{\supp(h)} g(s) \dx s \not= 0$ and $\int_{\supp(h)} g(s) \, \sgn h(s) \dx s \not= 0$.
	
	Note, however, that such a simple argument does not work if we do not assume $h$ to satisfy the estimate specified at the beginning of this remark.
\end{rem}

We close this section by pointing out that, while all the function spaces used above consist of complex-valued functions (in order for our Fourier transform arguments to work), the above results also tell us how to solve our integral equation~\eqref{eq:main} for real functions. This follows from the following remark:

\begin{rem} \label{rem:real-vs-complex}
	Let $E$ be a complex Banach space and suppose that $E$ is a complexification of a real Banach space $E_\R$ (for details about complexifications of Banach spaces, see for instance \cite{Munoz1999}). Let $\Gi: E \to E$ be a bounded linear operator which leaves $E_\R$ invariant and let $\Gi_\R: E_\R \to E_\R$ be the restriction of $T$ to $E_\R$. Then the following assertions hold:
	\begin{enumerate}
		\item[(a)] The operator $\Gi$ is injective if and only if $\Gi_\R$ is injective.
		\item[(b)] Assume that $\Gi$ is injective and let $v,w \in E$ such that $\Gi w = v$. If $v \in E_\R$, then $w \in E_\R$.
	\end{enumerate}
	The proofs are straightforward, so we omit them here.
\end{rem}

Note that the above remark applies in particular to the case where $E$ and $E_\R$ are the complex-valued and the real-valued $L^2$-spaces over $(\R^\times, \lvert x\rvert^c\dx x)$ and where $\Gi$ is the operator defined in Proposition~\ref{prop:operator-on-L_2}. Hence, Corollaries~\ref{cor:injectivit-and-surjectivity-of-G} and~\ref{cor:criterion-for-solvability} and Remark~\ref{rem:holomorphic-case} also give criteria for the solvability of our integral equation in the real case.


\section{Simulation study}\label{sec:simulation_study}
In this section, we provide numerical results for the pure jump infinitely divisible random fields given in Examples~\ref{ex:pure_jump_1}
and~\ref{ex:pure_jump_2}. In both examples, we used the method proposed in Remark~\ref{rem:criterion-for-convergence}, (b), in order to find a suitable value for $a_n$.
Unfortunately, the upper bound $e_n$ for the mean square error of $\widehat{uv_1}$ depends on the unknown function $uv_1$.
Nevertheless, in both examples, for some constant $C = C(uv_1, \Delta, d) > 0$, we have $\err_1(n) \leq C n^{-\frac{a}{2a + 1}} =: e_n$, $n \in \N$. On the other hand,
the number $\alpha_2$ can be chosen arbitrarily large in both cases; set $\alpha_2 = 2 \alpha_1$. 
Now, in order to assess the constant $C$, consider $k \in \N$ independent copies $\widetilde{uv_0}^{(1)}, \dots, \widetilde{uv_0}^{(k)}$ of 
$\widetilde{uv_0}$ and set
\begin{equation}\label{eq:fit_C}
C_k = n^{\frac{a}{2a + 1}} \Big( \argmin\limits_{a_n} \frac{1}{k} \sum_{j=1}^k \|\widetilde{uv_0}^{(j)} - uv_0\|_{L^2(\R^\times, \dx x)} \Big)^{2}.
\end{equation}
The $\argmin$ in~\eqref{eq:fit_C} is taken over all $a_n$ within the interval 
\begin{equation*}
\begin{cases}
[\min\limits_{x} |\mu_f(x)|, \max\limits_{x} |\mu_f(x)|], & \text{if} \ \min\limits_{x} |\mu_f(x)| > 0, \\
(0 , \max\limits_{x} |\mu_f(x)|], & \text{if} \ \min\limits_{x} |\mu_f(x)| = 0.
\end{cases}
\end{equation*}
A simple calculation shows that $L(a_n) := \|\widetilde{uv_0} - uv_0\|_{L^2(\R^\times, \dx x)}$ is a continuous function w.r.t.\ the parameter
$a_n$. Moreover, $\lim_{s \downarrow 0} L(s) > 0$ (possibly infinite), if $\min_{x} |\mu_f(x)| = 0$ whereas 
$\lim_{s \downarrow 0} L(s) = L(0)$ is finite in case that $\min_{x} |\mu_f(x)| > 0$. Since $L(s) = 0$ for any $s \geq \max_{x} |\mu_f(x)|$, 
$C_k$ thus is well-defined. Note that $\max_{x} |\mu_f(x)|$ is always finite due to integrability 
property~\eqref{eq:integrability-condition-stochastic-part}. Now, since $a_n = e_n^{1/2}$, $C = n^{\frac{a}{2a+1}} e_n$
and $e_n \geq \Big( 1+ \frac{L}{2\sqrt{\pi}} \Big)^{-2} \err_0(n)$ (cf. the proof of Corollary~\ref{cor:criterion-for-convergence}),
we set $L=0$ and thus $a_n = C_k^{1/2} \cdot n^{-\frac{a}{4a + 2}}$
for all $n \in \N$. \\ 

Due to high computation time (cf. Remark~\ref{rem:comp_time}), we used $k=10$ in our examples. For the 
parameter $l$ in~\eqref{eq:uv_1_estimator_pure_jump} we follow the recommendation in~\cite{ KarRothSpoWalk19} and use the values $l=1,2,3$.

\begin{rem}
	In order to determine $C_k$, the above method requires the a-priori knowledge of $uv_0$. This allows us to
	test our method and to find dimension dependent constants $C_k$ and $a_n$ which can be used in similar computations where $uv_0$ is not explicitly known. In the variety of test computations we performed, $a_n$ lies in the range from $0.5$ to $1.5$. Moreover, the unknown $uv_0$ can be also estimated via bootstrap out of the available data in order to assess the quantity $\err_0(n)$. 
\end{rem}

The computations in the following sections were performed on a CPU Intel Xeon E5-2630v3, 2.4 GHz with 128 GB RAM.


\subsection{Numerical results for Example~\ref{ex:pure_jump_1}}
Suppose $\theta = 4$ and let $(Y_j = X(j))_{j \in \{ -50, -49, \dots, 49 \}}$ 
be a sample drawn from $X$ (with $\Delta = 1$ and $n=100$). Since $uv_1 \in H^a(\R)$ for any $0 < a < 1$ one can fix e.g.
$a = 1/2$. Then, using $k=10$ in~\eqref{eq:fit_C} we obtain $C_k = 0.8$; hence, $a_n = 0.5$. Table~\ref{table:result_proc} shows mean and standard
deviation of the mean square errors of our estimates based on $100$ simulations for different values of $l$. The results for $l=2,3$ are quite 
similar and significantly better than for $l=1$. Comparing computation times (cf. Table~\ref{table:result_proc}) we therefore prefer to choose $l=2$. Figure~\ref{fig:traject_vs_est_proc} shows a trajectory of the process $X$ and the corresponding estimators $\widehat{uv_0}$ and $\widetilde{uv_0}$ 
with $l=2$.

\begin{table}[h]
	\centering
	\begin{tabular}{@{}rrccc@{}} \toprule
		&  & $l=1$  & $l=2$ & $l=3$  \\ \midrule
		\multirow{2}{1cm}{$\widehat{u v_0}$} & mean & 0.0408022914 & 0.0292780409 & 0.0242192599   \\
		& sd & 0.0077291047 & 0.0077990935 & 0.0071211386 \\ \midrule
		\multirow{2}{1cm}{$\widetilde{u v_0}$} & mean & 0.0195522346 & 0.0116813060 & 0.0094711975   \\
		& sd & 0.0059640590 & 0.0054764933 & 0.0047269966 \\ \midrule \midrule
		\multirow{2}{1cm}{comp. times} & mean & 51.27 & 63.97 & 66.96 \\
		& sd & 2.428181 & 2.952332 & 22.39901
		\\ \bottomrule
	\end{tabular}
	\vspace{0.2cm}
	\caption{Empirical mean and standard deviation of the mean square errors and the computation times (in seconds) of estimates 
		$\widehat{uv_0}$ and $\widetilde{uv_0}$ based on $100$ simulations ($d=1$).}
	\label{table:result_proc}
\end{table}

\begin{figure} [h!]
	\centering 
	\subfigure[trajectory of $X$]
	{\includegraphics[width=0.48\textwidth]{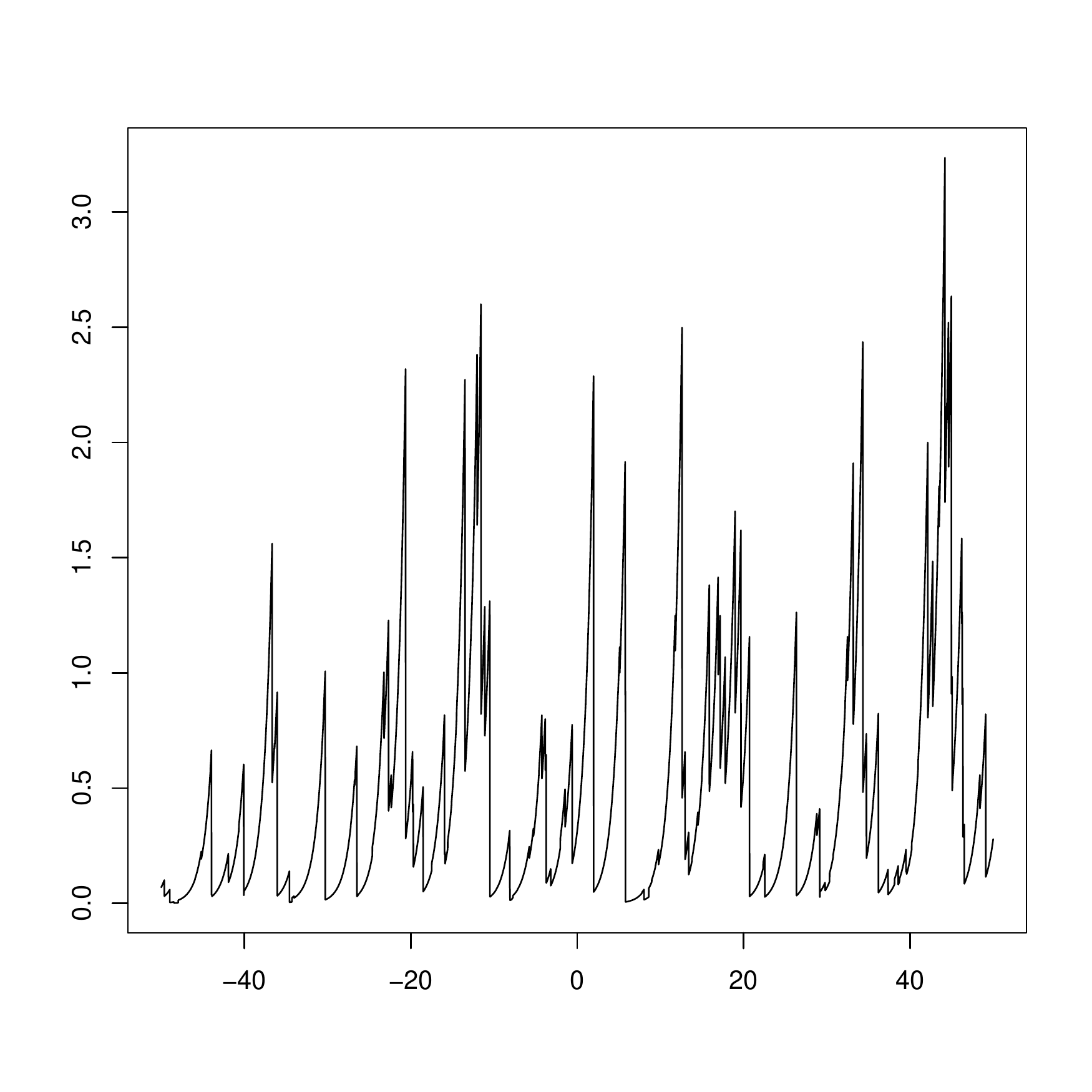} } 	
	\hfill
	\subfigure[$\widehat{u v_0}$ vs. true $u v_0$ (dashed line)]
	{\includegraphics[width=0.48\textwidth]{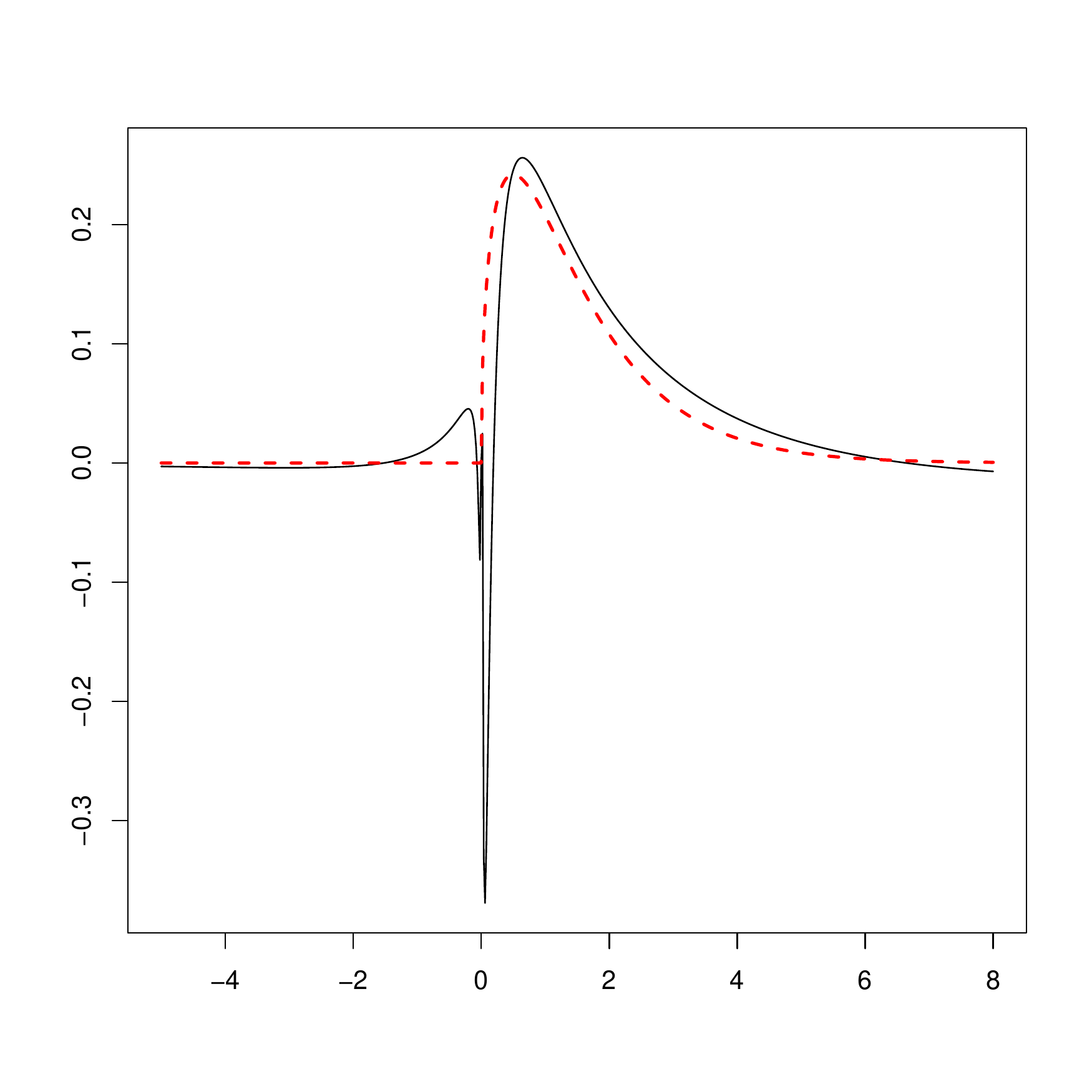} }
	\subfigure[$\widetilde{u v_0}$ vs. true $u v_0$ (dashed line)]
	{\includegraphics[width=0.48\textwidth]{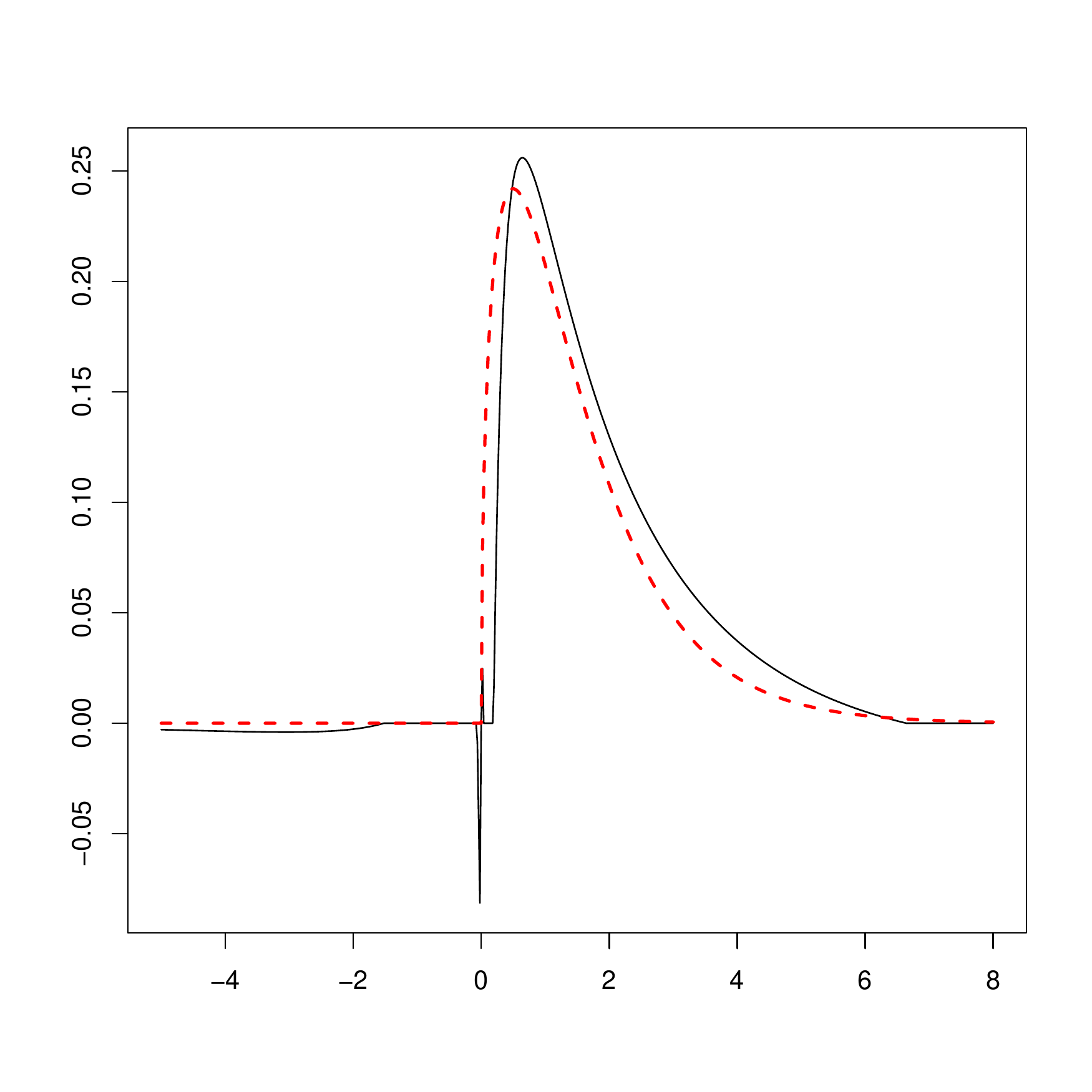} }
	\caption[]{Trajectory of the process $X(t)=\int_{t-\theta}^{t} e^{x-t} \Lambda(dx)$, $t \in \R$ from Example~\ref{ex:pure_jump_1} and the
		corresponding estimators $\widehat{u v_0}$ and $\widetilde{u v_0}$ with $a_n = 0.5$ and $l=2$, compared to the original
		$(uv_0)(x) = ( \frac{x}{\pi} )^{1/2} e^{-x} \one_{(0,\infty)}(x)$.} 
	\label{fig:traject_vs_est_proc}
\end{figure}

\subsection{Numerical results for Example~\ref{ex:pure_jump_2}}

Suppose $\kappa = 1$ and $\tau = 1/2$. Moreover, for $\Delta = 0.1$, let $(Y_j = X(\Delta j))_{j \in \{ -50, -49, \dots, 49 \}^2}$ 
be a sample drawn from $X$ (i.e. $n=10000$). Since $uv_1 \in H^a(\R)$ for any $0 < a < 1/2$ one can fix e.g.
$a = 1/4$. As in the previous example taking $k=10$ in~\eqref{eq:fit_C} leads to $C_k = 4.74$; consequently, $a_n = 1.01$. Mean and standard
deviation of the mean square errors of our estimates based on $100$ simulations for $l \in \{1,2,3\}$ are shown in Table~\ref{table:result_field}. 
Again, the mean square error for $l=1$ differs significantly 
from the mean square errors for $l=2,3$ whereas the values for $l=2,3$ are quite similar. For this reason, we prefer to use
$l=2$ due to shorter computation time (cf. Table~\ref{table:result_field}). Figure~\ref{fig:traject_vs_est_field} finally shows a trajectory of 
the field $X$ and the corresponding estimators $\widehat{uv_0}$ and $\widetilde{uv_0}$ with $l=2$.

\begin{figure} [h!]
	\centering 
	\subfigure[trajectory of $X$]
	{\includegraphics[width=0.48\textwidth]{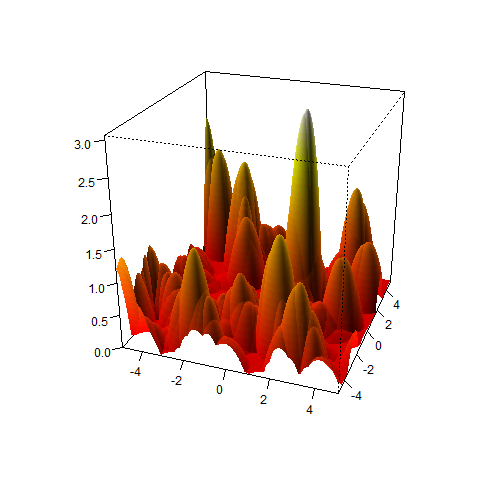} } 	
	\hfill
	\subfigure[$\widehat{u v_0}$ vs. true $u v_0$ (dashed line)]
	{\includegraphics[width=0.48\textwidth]{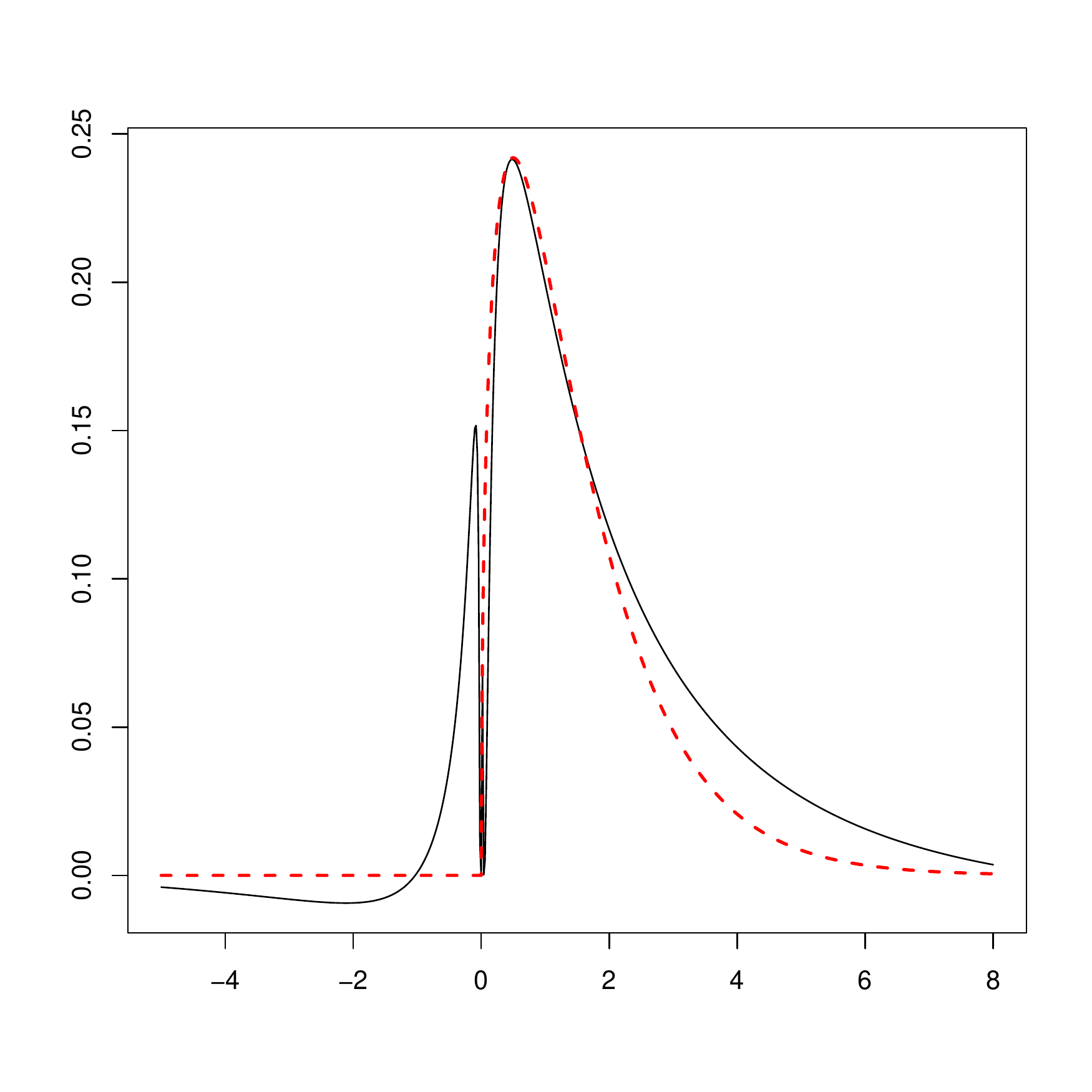} }
	\subfigure[$\widetilde{u v_0}$ vs. true $u v_0$ (dashed line)]
	{\includegraphics[width=0.48\textwidth]{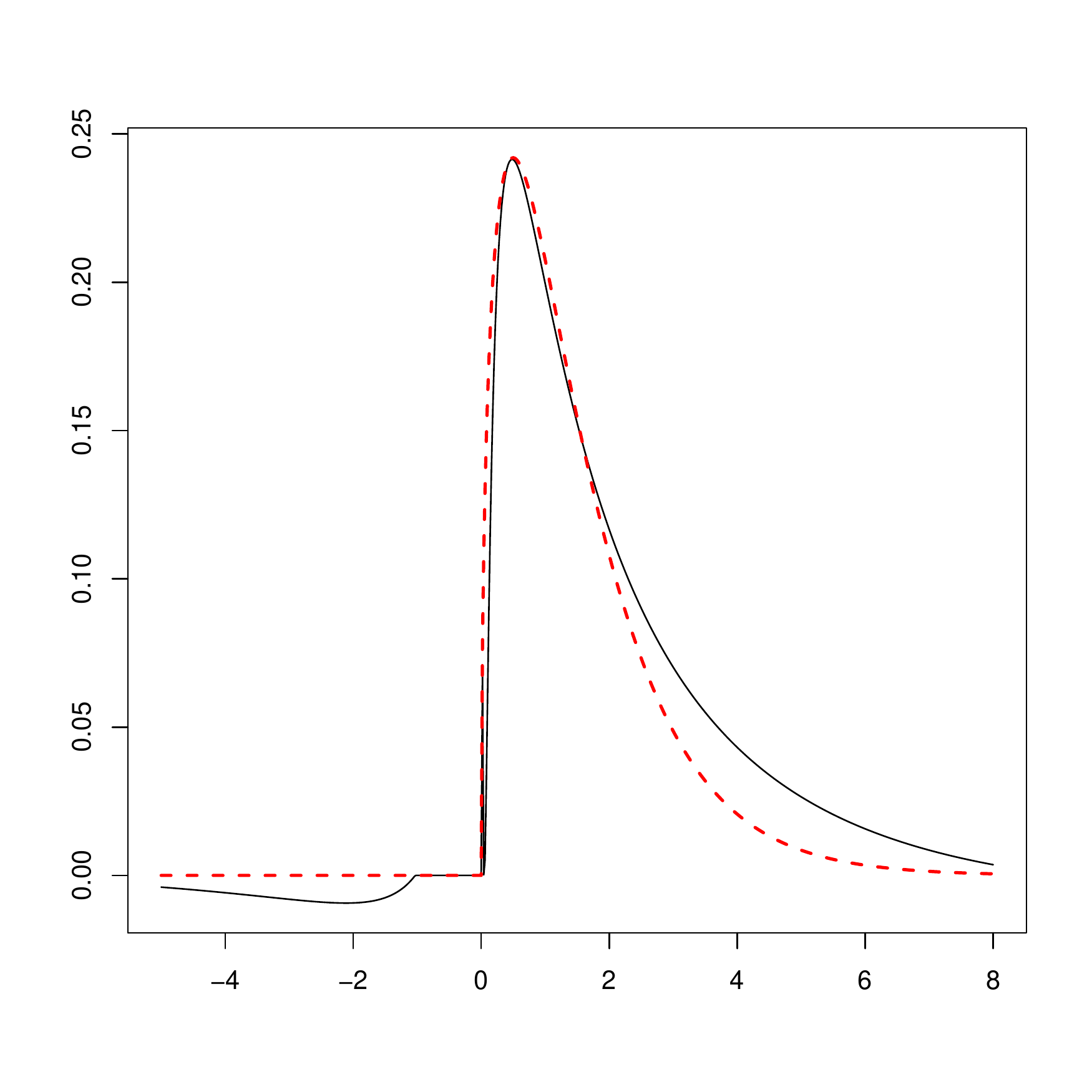} }
	\caption[]{Trajectory of the field $ X(t)=\int_{\|x-t\|_2 \leq \kappa} \frac{1}{2} (1 - \left\| x-t \right\|_2^2) \Lambda(dx)$, $t \in \R^2$ 
		from Example~\ref{ex:pure_jump_2} and the corresponding estimators $\widehat{u v_0}$ and $\widetilde{u v_0}$ with $a_n = 1.01$ and $l=2$, 
		compared to the original $(uv_0)(x) = ( \frac{x}{\pi} )^{1/2} e^{-x} \one_{(0,\infty)}(x)$.} 
	\label{fig:traject_vs_est_field}
\end{figure}

\begin{table}[h]
	\centering
	\begin{tabular}{@{}rrccc@{}} \toprule
		&  & $l=1$  & $l=2$ & $l=3$ \\  \midrule
		\multirow{2}{1cm}{$\widehat{u v_0}$} & mean & 0.0224887614 & 0.0113368827 & 0.0081844574   \\
		& sd & 0.0025617440 & 0.0024961707 & 0.0022603961 \\ \midrule
		\multirow{2}{1cm}{$\widetilde{u v_0}$} & mean & 0.0149609313 & 0.0068290548 & 0.0053098816   \\
		& sd & 0.0040406427 & 0.0030365831 & 0.0024361438 \\ \midrule \midrule
		\multirow{2}{1cm}{comp. times} & mean & 1003.71 & 3062.4 & 3834.87 \\
		& sd & 76.7873 & 211.6424 & 561.4358
		\\ \bottomrule
	\end{tabular}
	\vspace{0.2cm}
	\caption{Empirical mean and standard deviation of the mean square errors and the computation times (in seconds) of estimates 
		$\widehat{uv_0}$ and $\widetilde{uv_0}$ based on $100$ simulations ($d=2$).}
	\label{table:result_field}
\end{table}


%

\clearpage

\nocite{LevMattersIV,BelomGoldenschl19,BelomOrlPan19,BelomSchoenm16}

\bibliographystyle{plain} 

\end{document}